\documentclass[12pt]{article}
\usepackage{amssymb}
\usepackage{verbatim}
\usepackage{array}
\usepackage{latexsym}
\usepackage{enumerate}
\usepackage{amsmath}
\usepackage{amsfonts}
\usepackage{amsthm}
\usepackage{color}
\usepackage[english]{babel}

\title{\bf{Large $2$-groups of automorphisms of curves with positive $2$-rank}}
\date{}

\author{M.~Giulietti ${}^*$ and G.~Korchm\'aros
 \thanks{Research supported by  the Italian
    Ministry MIUR,  Aspetti geometrici, combinatorici e gruppali delle Geometrie di Galois, PRIN 2008}
}

\newtheorem{theorem}{Theorem}[section]
\newtheorem{proposition}[theorem]{Proposition}
\newtheorem{lemma}[theorem]{Lemma}

{\theoremstyle{definition}
\newtheorem*{definition*}{Definition}

\newtheorem{rem}[theorem]{Remark}
\newtheorem*{proposition*}{Proposition}
\newtheorem*{corollary*}{Corollary}
\newtheorem*{lemma*}{Lemma}

\def\cC{\mathcal C}
\def\cD{\mathcal D}
\def\cE{\mathcal E}

\def\cP{\mathcal P}

\def\cX{\mathcal X}
\def\cY{\mathcal Y}
\def\cZ{\mathcal Z}

\def\K{\mathbb{K}}

\def\ord{\mbox{\rm ord}}

\def\div{\mbox{\rm div}}
\def\min{{\rm min}}

\def\Tr{\mbox{\rm Tr}}

\def\gg{\mathfrak{g}}

\newcommand{\aut}{\mbox{\rm Aut}}

\newcommand{\ga}{\alpha}
\newcommand{\gb}{\beta}

\newcommand{\gep}{\epsilon}

\newcommand{\gs}{\sigma}

\newcommand{\ha}{{\textstyle\frac{1}{2}}}

\newcommand{\qa}{\textstyle\frac{1}{4}}
\newcommand{\tqa}{\textstyle\frac{3}{4}}
\newcommand{\teit}{\textstyle\frac{3}{8}}

\newcommand{\eit}{\textstyle\frac{1}{8}}

\sloppy

\begin{document}
\maketitle
\thanks{2000 {\em Math. Subj. Class.}: 14H37 }


    \begin{abstract}
Let $\K$ be an algebraically closed field of characteristic $2$,
and let $\cX$ be a curve over $\K$ of genus $\gg\ge 2$ and $2$-rank $\gamma>0$. For $2$-subgroups $S$ of the $\K$-automorphism group $\aut(\cX)$ of $\cX$, the Nakajima bound is $|S|\leq 4(\gg-1)$. For every $\gg=2^h+1\ge 9$, we construct a curve $\cX$ attaining the Nakajima bound and determine its relevant properties: $\cX$ is a bielliptic curve with $\gamma=\gg$, and its $\K$-automorphism group  has a dihedral $\K$-automorphism group of order $4(\gg-1)$ which fixes no point in $\cX$. Moreover, we provide a classification of $2$-groups $S$ of $\K$-automorphisms not fixing a point of $\cX$ and such that $|S|> 2\gg-1$. \end{abstract}
    \section{Introduction}
In the present paper, $\K$ is an algebraically closed field of
characteristic $2$,  $\cX$ is {a} (projective, non-singular,
geometrically irreducible, algebraic) curve of genus $\gg\geq
2$ and $2$-rank $\gamma$, $\aut(\cX)$ is the $\K$-automorphism group of $\cX$, and $S$
is a subgroup of $\aut(\cX)$ such that $|S|\geq 8$ is a power of $2$.

It is known that $S$ may be quite large compared to $\gg$. Stichtenoth \cite{stichtenoth1973I} proved that if $S$ fixes a point of $\cX$ then $|S|\leq 8\gg^2$. He also pointed out that his upper bound is attained by the non-singular model $\cX$ of the hyperelliptic curve of genus $2^{k-1}$ and equation $Y^2+Y+X^{2^k+1}=0$.

For $\gamma>0$, Stichtenoth's bound can be strengthened. Nakajima \cite{nakajima1987} showed indeed that if $\gamma>0$ then $|S|\leq 4(\gg-1)$.
The problem of finding curves attaining  Nakajima's bound is solved positively in Section \ref{biellipex}, see Theorem \ref{princel2}. For every $n=2^h\geq 8$ and
$n=\gg-1$, we determine such a curve $\cX$ with the following properties: $\cX$ is a bielliptic curve with $\gamma=\gg$ and it has a dihedral $\K$-automorphism group $S$ of order $4(\gg-1)$ which fixes no  point in $\cX$.

On the other hand, Lehr and Matignon \cite{lehr-matignon2005} observed that if $|S|>4(\gg-1)$ then $S$ fixes a point of $\cX$, see also \cite[Remark 11.128,\,Lemma 11.129]{hirschfeld-korchmaros-torres2008}.

The above results has given a motivation to investigate the possibilities for $\cX$, $\gg$ and $S$ when either $|S|$ is close to $8\gg^2$ (and $S$ fixes a point of $\cX$), or $|S|$ is close to $4(\gg-1)$ but $S$ fixes no point of $\cX$.

The first possibilities have recently been investigated by Lehr, Matignon and Rocher, see  \cite{lehr-matignon2005,matignon-rocher2008,rocher1,rocher2}. In \cite{lehr-matignon2005}, it is shown that $|S|\geq 4\gg^2$ only occurs when $\cX$ is the non-singular model of the Artin-Schreier
curve of equation $Y^q+Y+f(X){=0}$ with $f(X)=XP(X)+cX$ where $P(X)$ is an additive polynomial of $\K[X]$ and $q$ is a power of $2$.

To investigate the second possibility the hypotheses below are assumed:
\begin{itemize}
\item[(I)] $|S|\geq 8$ and $|S|>2(\gg-1)$,
\item[(II)] $S$ fixes no point on $\cX$.
\end{itemize}

Before stating our results we point out the prominent role of central involutions in this context.

Let $u$ be a central involution in $S$, that is an involution $u\in Z(S)$, and consider the associated quotient curve $\bar{\cX}=\cX/U$ where $U=\langle u \rangle$. The factor group $\bar{S}= S/U$ has order $\ha|S|$ and it is a $\K$-automorphism group of $\bar{\cX}$. Also, $\gg-1\geq 2(\bar{\gg}-1)$ where $\bar{\gg}$ is the genus of $\bar{\cX}$.  Therefore, either
\begin{itemize}
\item[(A)]$\bar{\gg}\leq 1$; or
\item[(B)] $\bar{\gg}=2$ and $|\bar{S}|=4$; or
\item[(C)] $\bar{\gg}\geq 2$, and hypothesis (I) is inherited by $\bar{S}$, viewed as a subgroup of $\aut(\bar{\cX)}$, but $\bar{S}$ fixes a point on $\bar{\cX}$; or
\item[(D)] $\bar{\gg}\geq 2$ and both hypotheses (I) and (II) are inherited  by $\bar{S}$, viewed as a subgroup of $\aut(\bar{\cX)}$.
\end{itemize}
If case (D) occurs then $u$ is called an {\em{inductive}} central involution of $S$. Note that,
if $|S|\geq 16$ and
no non-trivial element in $S$ fixes a point of $\cX$ then every central involution is inductive. It may happen that $\bar{S}$ also has an inductive central involution, say  $\bar{u}$.
In this case the quotient curve $\bar{\bar{\cX}}=\bar{\cX}/\langle \bar{u} \rangle$ with its inherited $\K$-automorphism group $\bar{\bar{S}}=\bar{S}/\langle \bar{u} \rangle$ satisfies both (I) and (II), as well. Therefore, an inductive argument can be used to go on as far as the resulting curve has an inductive central involution. Since the order of the inherited group halves at each step, after a finite number of steps a curve free from inductive central involutions is obtained. Such a finite sequence of curves is called an {\emph{inductive sequence}}.


Now, our results are stated.
\begin{theorem}
\label{princ} Let $\cX$ be a curve of genus $\gg\geq 2$ and $2$-rank $\gamma$ defined over an algebraically closed field $\K$ of characteristic $2.$ Assume that
$\aut(\cX)$ has a subgroup $S$ of order a power of $2$ such that both {\rm{(I)}} and {\rm{(II)}}  hold. If $S$ contains no inductive central involution then $\gg=\gamma$, and one of the following two cases occurs.
\begin{itemize}
\item[\rm(1)]  $|S|=4(\gg-1)$, $\cX$ is a bielliptic curve, and $S$ is a dihedral group.
\item[\rm(2)] $|S|=2\gg+2$, and $S=D\rtimes E$, the semidirect product of an elementary abelian group $D$ of index $2$ by a group $E$ of order $2$. If $S$ is abelian, then it is an elementary abelian group and $\cX$ is a hyperelliptic curve.
\end{itemize}
\end{theorem}
Theorem \ref{princ} is a corollary of the following result proven in Section \ref{dim1}.
\begin{theorem}
\label{princ1} Let $\cX$ be a curve of genus $\gg\geq 2$ and $2$-rank $\gamma$ defined over an algebraically closed field $\K$ of characteristic $2.$ Assume that
$\aut(\cX)$ has a subgroup $S$ of order a power of $2$ such that both {\rm{(I)}} and {\rm{(II)}}  hold. Then one of the following cases occurs:
\begin{itemize}
 \item[\rm(i)] $|S|=4(\gg-1)$, $\gamma=\gg$  and $\cX$ is a bielliptic curve. Furthermore, either
 \begin{itemize}
 \item[\rm(ia)] $S$ is dihedral and has no inductive central involution; or
\item[\rm(ib)] $S=(E\times \langle u \rangle)\rtimes \langle w \rangle$ where $E$ is cyclic group of order $\gg-1$ and $u$ and $w$ are involutions. The factor group $S/\langle u\rangle$ is a dihedral {group}, and the two involutions of $E\times \langle u \rangle$ are the unique two central inductive involutions of $S$.
 \end{itemize}
\item[\rm(ii)] $\gamma=\gg$, and $\rm{(2)}$ in Theorem \ref{princ} holds.
\item[\rm(iii)] Every central involution of $S$ is inductive.
\end{itemize}
\end{theorem}
In Section \ref{expex} we exhibit several examples showing that all cases occur. We also provide an explicit example illustrating an inductive sequence of curves.

\section{Bielliptic curves with a large dihedral auto- morphism group of order a power of $2$}
\label{biellipex}
{}Cyclic extensions of order a power of the characteristic of $\K$ are well known from the classical literature on function field theory, see \cite{albert1934,albert1935,witt1935,witt1936,lang2002}.

Here we briefly outline the general construction technique for such extensions when it is  applied to an elliptic function field.  Then we show that in some cases the resulting cyclic function field has a dihedral automorphism group with the properties described in case (1) of  Theorem \ref{princ}.
This requires some computational results given in the forthcoming subsection.

Let $\bar{\cX}$ be an elliptic curve with $2$-rank $\bar \gamma=1$. An affine equation of $\bar{\cX}$ is
\begin{equation}
\label{egye} f(x,y)=y^2+xy+x^3+\nu x^2+\mu
\end{equation}
where $\mu,\nu\in \K$ and $\mu\neq 0$.
Since $\bar{\gamma}=1$, the zero divisor class group ${\rm{Pic}}_0({\bar{\cX}})$ of $K(\bar{\cX})$
(isomorphic to the {group defined by the} point addition on $\bar{\cX}$), contains a
unique {cyclic} subgroup of order $2^m$ for every $m\geq 1$.
 Therefore,  for every $m\geq 1$, $\aut(\bar{\cX})$ has a cyclic subgroup ${\bf{C}}_n$ of order $n=2^m$ such that no non-trivial element of ${\bf{C}}_n$ fixes a point of $\bar{\cX}$. Let $g$ be a generator of ${\bf{C}}_n$.
\ \

There exists a cyclic extension $\cX$ of $\bar{\cX}$, and
all such cyclic extensions are obtained in the following way, see \cite[Section V]{witt1935}.

For $\xi\in \K(\bar{\cX})$, the relative $g$-trace of $\xi$ is defined to be
\begin{equation}
\label{trace22marzo}
 \Tr_g(\xi)=\xi+g(\xi)+\ldots + g^{2^{m}-1}(\xi).
 \end{equation}

Take an element $d\in \K(\cX)$ with $\Tr_g(d)=1$, and let $a=d^2+d$.
For $a\in \K(\bar{\cX})$ and $v=0,1,\ldots, n-1$, let
$$a_{g^0}=0, \quad \text{ and } \quad a_{g^v}=a+g(a)+\ldots+ g^{v-1}(a)\text{ for }v\ge 1.$$
Furthermore, take $c\in \K(\bar{\cX})$ with $\Tr_g(c)\neq 0$. Then
\begin{equation}
\label{eqwitt22marzo}
e=\frac{1}{\Tr(c)}\sum_{v=0}^{n-1}\,a_{g^{v}}\,g^v(c)
\end{equation}
satisfies equation ${g}(e)+e=a$; see
\cite[Section I]{witt1935}.
Here $e$ cannot be written as {$\zeta^2+\zeta$ with $\zeta$}$\in \K(\bar{\cX})$; see \cite[Section V]{witt1936}. Therefore, $\K(\cX)=\K(\bar{\cX})(z)$ with $z^2+z+e=0$ is an Artin-Schreier extension of $\K(\bar{\cX})$.
The map
\begin{eqnarray*}
h\,:\, (x,y,z)\to (g(x),g(y),z+d)
\end{eqnarray*}
is a $\K$-automorphism of $\cX$
whose order is equal to  $2n=2^{m+1}$.
Also, ${\bf{C}}_{2n}=\langle h\rangle $ preserves $\bar{\cX}$ and the $\K$-automorphism group ${\bf{C}}_{2n}/\langle w \rangle$ of $\bar{\cX}$ coincides with ${\bf{C}}_n$.

Now, consider the elliptic involution
\begin{equation}
\label{eq24marzo2011}
\varphi\,:\, (x,y) \to (x,x+y)
\end{equation}
which is a $\K$-automorphism of $\bar{\cX}$.
Since $\varphi g\varphi=g^{-1}$, $g$ together with $\varphi$ generate a $\K$-automorphism group $\bar{D}$ of $\bar{\cX}$ that is a dihedral group ${\bf{D}}_n$ of order $2n$.

The question arises whether $\varphi$ extends to an involutory $\K$-automorphism {$\psi$} of $\cX$ in such a way that the subgroup generated by {$\psi$} and ${\bf{C}}_{2n}$ is isomorphic to a dihedral group ${\bf{D}}_{2n}$ of order
$4n$.
If $\cX$ itself is an elliptic curve, then the answer is affirmative. Here we look
for non-elliptic curves to obtain examples for case {(1)}  of Theorem \ref{princ}.
\subsection{Some computations}
Let ${\bar{\cX}}$ {be} the elliptic curve over $\K$ with $2$-rank $\bar \gamma=1$ and affine  equation
$$
{\bar{\cX}}: y^2+xy+x^3+\mu=0.
$$
Fix a power $n$ of $2$, and let $g_0$ be a generator of the cyclic subgroup  of order $2n$ in the automorphism group of $\bar \cX$. Let $g=g_0^2$, and
 $\varphi$ be the elliptic involution defined by (\ref{eq24marzo2011}).
Let $\oplus$ denote the point addition on ${\bar{\cX}}$ such that the infinite point $Y_\infty$ is the neutral element of $({\bar{\cX}},\oplus)$. Also, let
$$
[i]P=\underbrace{P\oplus P\oplus \ldots \oplus P}_{i \text{ times }},
$$
and $\ominus P$ be the opposite of $P$ in $({\bar{\cX}},\oplus)$.
For a positive integer $r$, let $${\bar {\cX}}[r]=\{P\in {\bar{\cX}}\mid [r]P=Y_\infty\}.$$ As $\bar \gamma=1$, when $r$ is a power of $2$ the group ${\bar {\cX}}[r]$ is a cyclic group of order $r$.

{It will cause no confusion if we use the same letter to designate an automorphism of ${\bar{\cX}}$ and its pull-back.
In particular, $g_0^i$ will also denote a map acting on the points of ${\bar{\cX}}$ as follows:}
\begin{equation}\label{dividivi0}
g_0^i(P)=P\oplus [i]P_0.
\end{equation}
{Note that}
for each $\delta \in \K({\bar{\cX}})$
\begin{equation}\label{dividivi}
\div(\delta)=\sum_{P\in \bar{\cX}} n_P P \Rightarrow \div(g_0^i(\delta))=\sum_{P\in \bar{\cX}} n_P (P\oplus [2n-i]P_0).
\end{equation}
 Let $P_0=(w_1,w_2)$ be a generator of ${\bar{\cX}}[2n]$, that is,
$P_0$ is the point of ${\bar{\cX}}$ such  that $g_0(P)=P\oplus P_0$.

Let $\cP={\bar{\cX}}[n]$ and $\cZ={\bar{\cX}}[2n]\setminus {\bar{\cX}}[n]$.
Clearly $[2]P_0$ is a generator of $\cP$. Also, $\cP$ consists of points $[2j]P_0$ with $j=0,\ldots, n-1$, whereas $\cZ$ comprises points $[2j+1]P_0$ with $j=0,\ldots, n-1$.

{}From \eqref{dividivi0} we deduce  for $i=1,\ldots, 2n-1$ that
 \begin{eqnarray}\label{gizero}
g_0^i:\,
x'=
\frac{X_iy+X_ix^2+(X_i^2+Y_i)x}{(x+X_i)^2},\quad
y'=\frac{y+Y_i}{x+X_i}(x'+X_i)+x'+Y_i,
\end{eqnarray}
where $[i]P=(X_i,Y_i)$.
Since $\varphi g_0\varphi=g_0^{-1}$, $g_0$ together with $\varphi$ generate a dihedral group  of order $4n$.
\begin{lemma}\label{keybis}
Let $\delta$ be a $\K$-linear combination of some rational functions $g_0^i(x)$. Then $x\delta$ is a square in $\K({\bar{\cX}})$. In particular, each zero of $\delta$ has even multiplicity.
\end{lemma}
\begin{proof}
To prove that $x\delta$ is a square, it is enough to show that $xg_0^i(x)$ is  a square for each $i$.
Equation \eqref{gizero} yields
$$
xg_0^i(x)=\frac{ X_i xy+ X_i x^3+(X_i^2+Y_i) x^2}{ (x+X_i)^2}.
$$
As $xy+x^3=y^2+\mu$ we obtain
$$
xg_0^i(x)=\frac{ X_i(y^2+\mu)+(X_i^2+Y_i) x^2}{ (x+X_i)^2}=\left(\frac{\sqrt X_i(y+\sqrt \mu)+x\sqrt{X_i^2+Y_i}}{x+X_i}\right)^2.
$$
Since $x\delta$ is a square,
$$
ord_P(x)+ord_P(\delta)
$$
is even for every $P\in \bar \cX$. Since $ord_P(x)$ is always even, every zero of $\delta$ has even multiplicity.
\end{proof}

\begin{lemma}\label{tracciain} The divisor of $\Tr_g(x)$ is
$$
\div\left(\Tr_g(x)\right)=2\sum_{j=0}^{n-1}[2j+1]P_0-2 \sum_{j=0}^{n-1}[2j]P_0
$$
\end{lemma}
\begin{proof}
The pole divisor of $x$ is  $2Y_\infty$. Furthermore, from \eqref{dividivi} the pole divisor of $g^j(x)$ is $2([2n-2j]P_0).$ This proves that each element in $\cP$ is a pole of multiplicity $2$ of $\Tr_g(x)$. Moreover, no other point of ${\bar {\cX}}$ can be a pole of $\Tr_g(x)$. To prove that $P_0$ is a zero of $\Tr_g(x)$, note that $g^j(x)(P)=x(P\oplus [2j]P_0)$. Therefore,
\begin{eqnarray*}
\Tr_g(x)(P_0)&=&x(P_0)+x([3]P_0)+\ldots+x([2n-3]P_0)+x([2n-1]P_0)\\
&=& \sum_{i=1}^{(n-2)/2}(x([i]P_0)+x([2n-i]P_0))=\sum_{i=1}^{(n-2)/2}(x([i]P_0)+x(\ominus [i]P_0)).
\end{eqnarray*}
{}From $x(R)=x(\ominus R)$ for each affine point $R$, we obtain $\Tr_g(x)(P_0)=0$.
As $\Tr_g(x)$ is invariant under $g$, each point in $\cZ$ is a zero of $\Tr_g(x)$. By Lemma \ref{keybis} the multiplicity of a zero of $\Tr_g(x)$ is at least $2$. The assertion then follows from $|\cZ|=|\cP|$.
\end{proof}

\begin{lemma}\label{strano0} The point $\ominus P_0$ is a zero of $x+g(x)$ of multiplicity  $2$.
\end{lemma}
\begin{proof}
Arguing as in Lemma \ref{tracciain}, we can deduce that
the pole divisor of $x+g(x)$ is $2Y_\infty+2[2n-2]P_0$.
The point $\ominus P_0$ is a zero of $x+g(x)$ since
$$
(x+g(x))(\ominus P_0)=x(\ominus P_0)+x(\ominus P_0\oplus [2]P_0)=x(\ominus P_0)+x(P_0)=2x(P_0)=0.
$$
Similarly it can be shown that $[n-1] P_0$ is a zero of $x+g(x)$.
By Lemma \ref{keybis}, the zero divisor of $x+g(x)$ is $2\ominus P_0+2[n-1]P_0$.
\end{proof}

\begin{lemma}\label{strano} The point $\ominus P_0$ is a zero of $xg(x)+w_1^2$ of multiplicity  $4$.
\end{lemma}
\begin{proof}
Note that $(xg(x))(P)=x(P)\cdot x(P\oplus [2]P_0)$.
By straightforward computation
\begin{equation}\label{coord2P0}
[2]P_0=(\gs_1^2,\gs_2^2), \text{ with } \gs_1=w_1+\frac{\sqrt \mu}{w_1},\,\gs_2=\frac{\gs_1}{w_1}(w_1+w_2+\sqrt \mu)+w_1.
\end{equation}
Also, by \eqref{gizero},
$$
xg(x)+w_1^2=(h(x)+w_1)^2
$$
with
$$
h(x)=\frac{\gs_1(y+\sqrt \mu)+x(\gs_1^2+\gs_2)}{x+\gs_1^2}.
$$

Taking into account \eqref{coord2P0} we obtain
\begin{eqnarray*}
h(x)+w_1 &=& \frac{\gs_1(y+\sqrt \mu+\gs_1x)+x\gs_2+w_1x+\gs_1^2w_1}{x+\gs_1^2}\\ &=&\frac{\gs_1(y+\sqrt \mu+\gs_1x)+x(\frac{\gs_1}{w_1}(w_1+w_2+\sqrt \mu)+w_1)+w_1x+\gs_1^2w_1}{x+\gs_1^2}\\
&=&\frac{\gs_1}{w_1(x+\gs_1^2)}\left( w_1y+w_1\sqrt \mu+w_1\gs_1x+x(w_1+w_2+\sqrt \mu)+\gs_1w_1^2\right)\\
%
&=& \frac{\gs_1}{w_1(x+\gs_1^2)}\left( w_1y+x(w_1+w_1^2+w_2)+w_1^3\right).
\end{eqnarray*}
The line with equation  $w_1Y+X(w_1+w_1^2+w_2)+w_1^3=0$ is the tangent line of ${\bar{\cX}}$ at $\ominus P_0=(w_1,w_1+w_2)$. This line passes through $[2]P_0$. Therefore the divisor of $h(x)+w_1$ is
$$
2(\ominus P_0)-Y_\infty-[-2]P_0,
$$
whence, the divisor of $xg(x)+w_1^2$ equals $
4(\ominus P_0)-2Y_\infty-2([-2]P_0).
$
\end{proof}

For an element $\xi\in \K({\bar \cX})$ and for a non-negative integer $v$, let
$$
\xi_{g^v}=0 \quad \text{ for } v=0, \qquad \xi_{g^v}=\xi \quad\text{ for } v=1, \quad \text{ and } \quad \xi_{g^v}:=\sum_{i=0}^{v-1}g^i(\xi) \quad \text{ for } v\ge 2.
$$
\begin{lemma}\label{nuovissimo} Let $\xi\in \K({\bar \cX})$ be such that both $\Tr_g(\xi)=0$ and $\varphi(\xi)=\xi$ hold. Then
\begin{itemize}
\item[{\rm{(i)}}] $\xi_{g^{v_1}}=\xi_{g^{v_2}}$ when $v_1\equiv v_2 \pmod n$;

\item[{\rm{(ii)}}] $\xi_{g^{v_1}}+\xi_{g^{v_2}}=g^{v_1}(\xi_{g^{v_2-v_1\pmod n}})$;

\item[{\rm{(iii)}}] $\varphi(\xi_{g^v})=\xi_{g^{-v+1\pmod n}}+\xi$.

\end{itemize}
\end{lemma}
\begin{proof}
As $\Tr_g(\xi)=\xi_{g^n}$, assertion (i) easily follows from $\Tr_g(\xi)=0$. To prove (ii),
$v_1<v_2$ may be assumed, as the case $v_1\ge v_2$ can be  prevented by replacing  $v_2$  with $v_2+hn$ for a sufficiently large  positive integer $h$). Then
$$
\xi_{g^{v_1}}+\xi_{g^{v_2}}=\sum_{i=v_1}^{v_2-1}g^i(\xi)=g^{v_1}\Big(\sum_{i=0}^{v_2-v_1-1}g^i(\xi)\Big)=g^{v_1}(\xi_{g^{v_2-v_1}}).
$$

If $v=0,1$, then (iii) clearly holds. We compute $\varphi(\xi_{g^v})$ for $v\ge 2$. By  $\varphi(\xi)=\xi$ and $\varphi g=g^{-1}\varphi$,
$$
\varphi(\xi_{g^v}) = \sum_{i=0}^{v-1}g^{n-i}(\xi)=\sum_{j=n-v+1}^ng^j(\xi).
$$
{}From $\Tr_g(\xi)=0$,  $$\varphi(\xi_{g^v})=\sum_{j=0}^{n-v}g^j(\xi)+g^n(\xi)=\xi_{g^{n-v+1}}+\xi=\xi_{g^{-v+1\pmod n}}+\xi,$$ when the assertion follows. \end{proof}
For an odd $k$ with $1\le k \le 2n-1$, define, as in Witt's paper \cite{witt1936}:
\begin{equation}
\label{eq18marzo}
d=\frac{x}{\Tr_g(x)},\qquad a=d^2+d.
\end{equation}
Furthermore, let
\begin{equation}
\label{eq22marzo}
c_k=g_0^k(x),\qquad e_k=\frac{1}{\Tr_g(c_k)}\sum_{v=0}^{n-1} a_{g^v}g^v(c_k).
\end{equation}
A straightforward computation gives the following result:
\begin{equation}
\label{witthilbert}
g(e_k)+e_k=a.
\end{equation}
Our purpose is to show that $\varphi(e_k)+e_k=a$ also holds, see Proposition \ref{21feb5} below. This requires some more computation.

\begin{proposition}\label{squareek} The rational function $e_k$ is a square.
\end{proposition}
\begin{proof} By Lemma  \ref{keybis} $d=\frac{x}{\Tr_g(x)}=\frac{x^2}{x\Tr_g(x)}$ is a square. Therefore $a_{g^v}$ is a square for each $v$. Then, we only need to show that
$\Tr_g(c_k)\cdot g^v(c_k)$ is a square for each $v$. This follows from the fact that
$$
x^2\cdot \Tr_g(c_k)\cdot g^v(c_k)=(x\Tr_g(g_0^k(x)))(xg_0^{2v+k}(x))=\Big(x\sum_{i=0}^{n-1}g_0^{2i+k}(x)\Big)(xg_0^{2v+k}(x))
$$
is a square by Lemma \ref{keybis}.
\end{proof}
\begin{lemma}\label{lemmadia} For the rational function $a$, both $\Tr_g(a)=0$ and $\varphi(a)=a$ hold.
\end{lemma}
\begin{proof}
We have that  $\Tr_g(a)=\Tr_g(d^2+d)=\Tr_g(d)^2+\Tr_g(d)=1+1=0$.
Moreover, from $\varphi g=g^{-1}\varphi$ it follows that $\varphi(\Tr_g(x))=\Tr_g(x)$. Therefore $\varphi(d)=d$ and $\varphi(a)=a$.
\end{proof}
\begin{lemma}\label{21feb1} $\Tr_g(ag_0(x))=\Tr_g(g(a)g_0(x))$.
\end{lemma}
\begin{proof}
Since $\Tr_g(g(a)g_0(x))=\Tr_g(g(ag_0^{-1}(x)))=\Tr_g(ag_0^{-1}(x))$, an equivalent formulation of the statement is
$$
\Tr_g(a(g_0(x)+g_0^{-1}(x)))=0.
$$
Let $\xi=x/(x+w_1)$.
By a straightforward computation
$$
g_0(x)+g_0^{-1}(x)= \frac{xw_1}{x^2+w_1^2}=\xi^2+\xi.
$$
Taking into account that
$a=d^2+d$ with $d=x/\Tr_g(x)$ we have
\begin{eqnarray*}
\Tr_g(a(g_0(x)+g_0^{-1}(x)))&=& \Tr_g\left((d^2+d)(\xi^2+\xi)\right) \\
 &=&\frac{1}{\Tr_g(x)^2}\left(
\Tr_g(x)\Tr_g\left(x(\xi^2+\xi)\right)+\Tr_g\left(x^2(\xi^2+\xi)\right)
\right).
\end{eqnarray*}
Note that $x(\xi^2+\xi)=w_1\xi^2$ and that $x^2(\xi^2+\xi)=w_1x+w_1^2(\xi^2+\xi)$. Also
\begin{equation}\label{17mar1}
\Tr_g(\xi^2+\xi)=\Tr_g(g_0(x)+g_0^{-1}(x))=
2\Tr_g(g_0(x))=0.
\end{equation}
Therefore, we need to show that
$\Tr_g(x)\cdot \Tr_g\left(w_1\xi^2\right)=\Tr_g\left(w_1x\right),
$
which is clearly equivalent to $\Tr_g(\xi)=1$. By \eqref{17mar1}, $\Tr_g(\xi)^2+\Tr_g(\xi)=0$, whence either $\Tr_g(\xi)=1$ or $\Tr_g(\xi)=0$. To prove that the latter case cannot occur
we show that $\Tr_g(\xi)(\ominus P_0)=1$.
Note that
$$
(g^2(\xi)+\ldots+g^{n-1}(\xi))(\ominus P_0)=\xi([3]P_0)+\xi([5]P_0)+\ldots+\xi([2n-5]P_0)+\xi([2n-3]P_0).
$$
As $\xi$ only depends on $x$, and   $x([i]P_0)=x([2n-i]P_0)$, we have
$$
\xi([i]P_0)+\xi([2n-i]P_0)=0,
$$
whence $(g^2(\xi)+\ldots+g^{n-1}(\xi))(\ominus P_0)=0$. Thus,  $\Tr_g(\xi)(\ominus P_0)=(\xi+g(\xi))(\ominus P_0)$.
By a straightforward computation
$$
\xi+g(\xi)=w_1\frac{x+g(x)}{(x+w_1)(g(x)+w_1)},
$$
whence
$$
\frac{1}{\xi+g(\xi)}=\frac{1}{w_1}\cdot \frac{(xg(x)+w_1^2)+w_1(x+g(x))}{x+g(x)}=1+ \frac{1}{w_1}\cdot \frac{xg(x)+w_1^2}{x+g(x)}.
$$
By Lemmas \ref{strano0} and \ref{strano}, $\ominus P_0$ is a zero of
 $\frac{xg(x)+w_1^2}{x+g(x)}$. Thus $(\xi+g(\xi))(\ominus P_0)=1$, and the proof is completed.
\end{proof}

\begin{lemma}\label{21feb2} For each odd $k$ with $1\le k \le 2n-1$, $\Tr_g(ag_0^k(x))=\Tr_g(g^k(a)g_0^k(x))$.
\end{lemma}
\begin{proof}
As $k$ is odd, $g_0^k$ is a generator of $\langle g_0\rangle$. Therefore, by Lemma \ref{21feb1}, we have
$\Tr_{g^k}(\bar ag_0^k(x))=\Tr_{g^k}(g^k(\bar a)g_0^k(x))$, where $\bar a=(x/\Tr_{g^k}(x))^2+x/\Tr_{g^k}(x)$.
But clearly $\Tr_{g^k}$ coincides with $\Tr_g$. Thus, $\bar a=a$ and
$$
\Tr_{g}(ag_0^k(x))=\Tr_{g}(g^k(a)g_0^k(x)).
$$
also holds.
\end{proof}

\begin{lemma}\label{21feb3} For each odd $k$ with $1\le k \le 2n-1$,
$$
\Tr_g(g_0^k(x)\cdot (a+g(a)+\ldots+g^k(a)))=0.
$$
\end{lemma}
\begin{proof} It is by induction on $k$. The assertion for $k=1$ is just Lemma \ref{21feb1}.
Now assume that
$$
\Tr_g(g_0^{k-2}(x)\cdot (a+g(a)+\ldots+g^{k-2}(a)))=0.
$$
Applying $g$ to the argument of $\Tr_g$ gives
$$
\Tr_g(g_0^{k}(x)\cdot (g(a)+g^2(a)+\ldots+g^{k-1}(a)))=0.
$$
By Lemma \ref{21feb2} and the additivity of $\Tr_g$,  the assertion follows.
\end{proof}

\begin{proposition}\label{21feb5} For each odd $k$ between $1$ and $2n-1$,
$$\phi(e_k)+e_k=a$$
\end{proposition}
\begin{proof}
It is straightforward to show that $\varphi(\Tr_g(c_k))=\Tr_g(c_k)$. Therefore, by (iii) of Lemma \ref{nuovissimo} and Lemma \ref{lemmadia},
$$
\phi(e_k)+e_k=\frac{1}{\Tr_g(c_k)}\left(\sum_{v=0}^{n-1} a_{g^v}g_0^{2v+k}(x)+\sum_{v=0}^{n-1}(a+a_{g^{-v+1}})g_0^{-2v-k}(x)\right).
$$
Let $t(v)=-v-k\pmod n$. Then $g_0^{-2t(v)-k}(x)=g_0^{2v+k}(x)$. Therefore
$$
\phi(e_k)+e_k=\frac{1}{\Tr_g(c_k)}\left(\sum_{v=0}^{n-1}(a_{g^v}+a+a_{g^{-t(v)+1}})g_0^{2v+k}(x))\right).
$$
By (ii) of Lemma \ref{nuovissimo},
\begin{eqnarray*}
\phi(e_k)+e_k &=& \frac{1}{\Tr_g(c_k)}\left(\sum_{v=0}^{n-1}(a+g^v(a_{g^{-t(v)+1-v\pmod n}}))g_0^{2v+k}(x))\right)\\
& = & \frac{a\Tr_g(c_k)}{\Tr_g(c_k)}+ \frac{1}{\Tr_g(c_k)}\left(\sum_{v=0}^{n-1}g^v(a_{g^{k+1}}g_0^k(x))\right)\\
& = & a+\frac{1}{\Tr_g(c_k)} \Tr_g((a+g(a)+\ldots+g^k(a))\cdot g_0^k(x)).
\end{eqnarray*}
The claim then follows by Lemma \ref{21feb3}.
\end{proof}

\subsection{Proof of the existence}
\label{exis22marzo}
We are in a position to show the existence of curves which provide examples for case (1) of Theorem \ref{princ}.

For this purpose, we consider the Artin-Schreier extension $\cX_k$ of $\bar{\cX}$ defined by the equation $z^2+z+e_k=0$, where $k$ is an odd integer with $1\leq k \leq 2n-1$.

We first construct some automorphisms of $\cX_k$.
Every element
in $\K(\cX_k)$ can uniquely be written as $(a_1+a_2y)z+a_3y+a_4$ with $a_1,a_2,a_3,a_4\in \K(x)$.  Furthermore, the map
\begin{equation}
\label{eqrho}
\rho\,:\, (x,y,z)\to (g(x),g(y),z+d)
\end{equation}
is a $\K$-automorphism of $\cX_k$. From $\Tr_g(d)=1$ we have that
\begin{equation}
\label{eqiota}
\iota=\rho^{n}\,:\,(x,y,z)=(x,y,z+1).
\end{equation}
 Therefore, $\iota$ is an involution, $\bar{\cX}=\cX_k^{\iota},$ and $\rho$ generates a cyclic subgroup
${\bf{C}}_{2n}$ of $\aut(\cX_k)$ of order $2n$. Also, ${\bf{C}}_{2n}$ preserves $\cX_k$ and the $\K$-automorphism group ${\bf{C}}_{2n}/\langle \iota \rangle$ of $\bar{\cX}$ coincides with the cyclic group of order $n$ generated by $g$.

A straightforward computation involving Proposition \ref{21feb5} gives the following result.
\begin{lemma}
\label{lemwitt1}
The map $$\psi\,:\,(x,y,z)\to (\varphi(x),\varphi(y),z+d)$$ is a $\K$-automorphism of $\cX_k$.
\end{lemma}
Next, the structure of the group generated by $\rho$ and $\psi$ is described.
\begin{proposition}
\label{propwitt} The group $S$ generated by $\rho$ and $\psi$ is isomorphic to
 $ {\bf{D}}_{2n}$.
\end{proposition}
\begin{proof} As $\varphi(d)=d$, both $\psi$ and $\psi \rho$ are involutions showing that $S\cong {\bf{D}}_{2n}$.
\end{proof}
To prove the theorem below it remains to show that {$\cX_k$} is non-elliptic for some odd $k$ with $1\le k \le 2n-1$.

For this purpose, the following results on the pole divisor of $e_k\in \K({\bar{\cX}})$ is useful.
\begin{lemma}\label{polesek} Let $k$ be an odd {integer} with $1\leq k \leq 2n-1$. Then
\begin{itemize}
\item[{\rm{(i)}}] every pole of $e_k$ belongs to $\cP\cup \cZ$;

\item[{\rm{(ii)}}] the point $Y_\infty$ is not a pole of $e_k$;

\item[{\rm{(iii)}}] for every $P\in \cZ$,   $v_{P}(e_k)\ge -4$;

\item[{\rm{(iv)}}] $v_{[-k]P_0}(e_k)\ge -2$ and equality holds provided that $[-k]P_0$ is not a zero of $$\sum_{v=1}^{n-1}(x^2+g(x^2)+\ldots+g^{v-1}(x^2))g^v(c_k).$$
\end{itemize}
\end{lemma}
\begin{proof}
Since each $g^i$ fixes $\Tr_g(x)$, $e_k$ can be written as
$$
e_k=\frac{1}{\Tr_g(c_k)\Tr_g(x)^2}\sum_{v=1}^{n-1}\left(x^2+x\Tr_g(x)+\ldots+g^{v-1}\left(x^2+x\Tr_g(x)\right)\right)g^v(c_k).
$$
Put $f_k=\sum_{v=1}^{n-1}\left(x^2+x\Tr_g(x)+\ldots+g^{v-1}\left(x^2+x\Tr_g(x)\right)\right)g^v(c_k)$.
Every point in $\cZ$ is a zero of $\Tr_g(x)$ with multiplicity $2$, and hence is a pole of $\Tr_g(c_k)$ with multiplicity 2. Therefore,
\begin{equation}\label{divfrac}
\div\bigg(\frac{1}{\Tr_g(c_k)\Tr_g(x)^2}\bigg)=\sum_{P\in \cZ}-2P+\sum_{P\in \cP}2P.
\end{equation}
\begin{itemize}
\item[{\rm{(i)}}] Note that $f_k$ is a $\K$-linear combination of products of functions $x$ and $g_0^i(x)$, with $1\le i \le 2n$. The poles of any of these functions are contained in $\cP\cup \cZ$. Taking into account \eqref{divfrac}, the claim follows.

\item[{\rm{(ii)}}] Note that $v_{Y_\infty}(g^v(c_k))\ge 0$ for any integer $v$, whence
\begin{equation}\label{19mar11}
v_{Y_\infty}(f_k)\ge \min\{v_{Y_\infty}(g^{i}\left(x^2+x\Tr_g(x)\right)\mid i=0,\ldots, n-2\}.
\end{equation}
As $x^2+x\Tr_g(x)=x(g(x)+g^2(x)+\ldots+g^{n-1}(x))$, we have
$$
v_{Y_\infty}(g^{i}\left(x^2+x\Tr_g(x)\right)=v_{Y_\infty}(g^i(x))+v_{Y_\infty}(g^{i+1}(x)+g^{i+2}(x)+\ldots+g^{i+n-1}(x)).
$$
The point $Y_\infty$ is a pole of $g^j(x)$ only when $j=0\pmod n$. In this case, $v_{Y_\infty}(g^j(x))=-2$ holds. Therefore $v_{Y_\infty}(g^{i}\left(x^2+x\Tr_g(x)\right)\ge -2$, and hence $v_{Y_\infty}(f_k)\ge -2$ holds by \eqref{19mar11}. By  \eqref{divfrac}, the claim follows.

\item[{\rm{(iii)}}]
Let $P=[iP_0]$, with $i$ odd. Clearly $[i]P_0$ is not a pole of any $g^j(x^2+x\Tr_g(x))$. For $v=1,\ldots,n-1$,  $[i]P_0$ is a pole of $g^v(c_k)$ precisely when $v \equiv n-((k+i)/2)\pmod{n}$. As the multiplicity of $[i]P_0$ as a pole of $g^v(c_k)$ is at most $2$, the  claim follows.
\item[{\rm{(iv)}}] Arguing as in (iii),  $v_{[-k]P_0}(e_k)\ge -2$ can be shown. In fact, $[-k]P_0$ is never a pole of $g^v(c_k)$, as $v\equiv 0 \pmod n$ does not occur. Note that $[-k]P_0$ is a zero of $g^j(x\Tr_g(x))$ for each $j$. Therefore,
$$
f_k([-k]P_0)=0 \Leftrightarrow \left(\sum_{v=1}^{n-1}\left(x^2+\ldots+g^{v-1}\left(x^2\right)\right)g^v(c_k)\right)([-k]P_0)=0.
$$
Taking into account \eqref{divfrac}, the claim follows.
\end{itemize}
\end{proof}

\begin{proposition}\label{generegiusto} Assume that there exist some point $P\in \bar{\cX}$ with such that the order of $e_k$ at $P$ is equal to $-2$. Then $\iota$ fixes exactly $n$ places of $\cX$, and $\cX_k$ has genus $n+1$. Also, the $2$-rank of $\cX_k$ is equal to $n+1$.
\end{proposition}
\begin{proof}
Let $\cY_k$ be a non-singular model of $\cX_k$, so that ${\mathbf D}_{2n}$ can be viewed as an automorphism group of $\cY_k$. Let $\pi:\cY_k\to {\bar{\cX}}$ denote the covering of degree $2$ associated with the function field extension $\K(\cY_k):\K(\cY_k)^\iota$.
By the Hurwitz genus formula applied to $\pi$, the
genus of $\cX_k$ is equal to $1+\frac{1}{2}\sum d_Q$, where, as usual, $d_Q$ denotes the different exponent of a point $Q$ of $\cY_k$ with respect to $\pi$ (see e.g. \cite[Proposition 3.7.8]{STI2ed}).
Let $\cE$ be the set of points $Q$ of $\cY_k$ such that
$d_Q>0$.
As $\cE$ coincides with the set of points fixed by $\iota$,  $\cE$ is preserved by ${\bf{C}}_{2n}$. More precisely, as ${\bf{C}}_{2n}/\langle \iota \rangle$ coincides with $\langle g\rangle$, the set $\cE$ consists of the points of $\cY_k$ lying over a $g$-invariant set of points $\mathcal D$ of ${\bar{\cX}}$.
By  \cite[Proposition 3.7.8]{STI2ed} each point in $\cD$ is a pole of $e_k$. By Lemma \ref{polesek}, either $\cD$ is empty, or $\cD=\cZ$.
Under our assumption, we prove that the former case cannot actually occur.
Let $t$ be a local parameter at $P$. By Proposition \ref{squareek},
$$
e_k=\gs t^{-2}+e_k'
$$
with $\gs \in \K, \,\, \gs \neq 0$ and $v_{P}(e_k')=0$.
Then clearly
$$
v_{P}(e_k+(\sqrt \gs/t)^2+(\sqrt \gs/t))=-1.
$$
By
 \cite[Proposition 3.7.8(c)]{STI2ed}, $P$ is totally ramified, and if $P'$ denotes the only point in $\cX_k$ lying over $P$, then the different exponent $d_{P'}$ is equal to $2$. This proves that $\cD=\cZ$. Now let $R\in \cZ$ be such that  $v_{R}(e_k)\neq -2$, and let $R'$ be the only point in $\cE$ such that $\pi(R')=R$.  By (iii) of Lemma \ref{polesek}, $v_{R}(e_k)= -4$ holds. Then, by \cite[Proposition 3.7.8]{STI2ed}, either
$d_{R'}=4$  or $d_{R'}=2$. If $d_{R'}=4$, then there exists $\xi \in \K{(\bar{\cX}})$ with $v_{R}(e_k+\xi^2+\xi)=-3$. But this
is impossible $e_k$ being a square.
Therefore, for each $Q\in \cE$ we have $d_{Q}=2$. Finally, as the size of $\cE$ is $n$, from the Hurwitz genus formula the genus of $\cX_k$ is equal to $n+1$. The Deuring-Shafarevich formula, see \eqref{eq2deuring},  applied to $S=\langle \iota \rangle$ shows that the $2$-rank of $\cX_k$ is equal to $n+1$ as well.
\end{proof}

\begin{proposition}
There exists a $k$ for which $\bar P=[-k]P_0$ is not a zero of $$\sum_{v=1}^{n-1}(x^2+g(x^2)+\ldots+g^{v-1}(x^2))g^v(c_k),$$
\end{proposition}
\begin{proof}
Let $\zeta_v=x+g(x)+\ldots+g^{v-1}\left(x\right)$,
and consider the rational function $\epsilon$ defined as follows:
$$
\gep(P)=\zeta_1(P)^2\cdot x([2]P_0)+\zeta_2(P)^2\cdot x([4]P_0)+\ldots +\zeta_{n-1}(P)^2\cdot x([2n-2] P_0).
$$
As $g^v(g_0^k(x))(\bar P)$ is  the $x$-coordinate of $ [2v]P_0$,
$$
\gep(\bar P)=\left(\sum_{v=1}^{n-1}\left(x^2+\ldots+g^{v-1}\left(x^2\right)\right)g^v(c_k)\right)(\bar P).
$$

Therefore, to prove the existence of  a suitable $k$
it will be enough to show that $\epsilon$
 has less than $n$ distinct zeros in $\mathcal Z$. Note that the values of $x([2v]P_0)$ are independent of $P$, and therefore can be viewed as constants. Let $\alpha_v\in \K$ be the square root of $x([2v] P_0)$. Then $\epsilon(P)=\theta(P)^2$, where
$$
\theta(P)=\zeta_1(P)\cdot \alpha_1+\zeta_2(P)\cdot \alpha_2+\ldots +\zeta_{n-1}(P)\cdot \alpha_{n-1}.
$$
We will prove that
 $\theta$ has less than $n$ distinct zeros in $\mathcal Z$. Expanding $\zeta_i(P)$ gives
\begin{eqnarray*}
\theta(P)&=&x(P)\alpha_1+(x(P)+ x(P\oplus [2] P_0))\alpha_2+\\ & & +(x(P)+ x(P\oplus [2] P_0)+x(P \oplus [4]P_0))\alpha_3+\\
& &+\ldots+(x(P)+\ldots+x(P\oplus [2n-4] P_0))\alpha_{n-1}.
\end{eqnarray*}
Note that $\alpha_{n/2}=0$ and that $\alpha_v=\alpha_{n-v}$. This depends on $[n]P_0=(0,\sqrt \mu)$  and on $[2v]P_0$ being the opposite of $[2n-2v]P_0$.  Therefore
$\alpha_1+\ldots+\alpha_{n-1}=0$, and hence there exist constants $\beta_i\in \K$ with
$$
\theta(P)=x(P\oplus [2]P_0)\cdot \beta_1+\ldots+x(P\oplus [2n-4]P_0))\cdot \beta_{n-2}.
$$
As $x(P\oplus [2i]P_0)=g^i(x)(P)$,  $\theta$ is a linear combination of some $g^i(x)$'s.
Clearly,
the only poles of $\theta$ are  points $[2n-2v]P_0$ for which $\beta_v\neq 0$. Each of those poles has multiplicity $2$. Then the number of zeros of $\theta$ is at most $2(n-2)$. By Lemma \ref{keybis}, each zero of $\theta$ has even multiplicity. If  $[-k]P_0$ is a zero of $\theta$ for each $k$, then the number of zeros is  larger than $2n-4$, which is a contradiction.
\end{proof}

Taking into account (iv) of Lemma \ref{polesek} together with Proposition \ref{generegiusto}, this ends the proof of the following result.
\begin{theorem}
\label{princel2} For every $n=2^h\ge 8$, some of the above bielliptic curves {$\cX_k$}  is of genus $\gg=n+1\geq 2$ and it has a dihedral $\K$-automorphism group $S$ such that $|S|=4(\gg-1)$. Furthermore, $\gamma=\gg$ and the (unique) central involution in $S$ fixes some points of $\cX$ and hence it is not inductive.
\end{theorem}
\subsection{Some more examples} {}From Theorem \ref{princel2}, the question arises whether curves other than $\cX_k$ can provide examples  for case (1) of Theorem {\ref{princ}}. To construct such a curve, a different choice for $d$ in (\ref{eq18marzo}) is necessary. The possibilities are described in the following result.
\begin{lemma}
\label{lemwitt} For $d\in\K(\bar{\cX})$ with $\Tr_g(d)=1$, $a=d^2+d$, $c\in \K(\bar{\cX})$ with $\Tr_g(c)\neq 0$, let $e$ be as defined in \eqref{eqwitt22marzo}.
Assume that  $\varphi(a)=a$ and $\varphi(c)=g(c)$. Then $\varphi(e)+e=a$, and either
\begin{itemize}
\item[\rm(i)] $d=\frac{x}{\Tr_g(x)}+\delta,$ or
\item[\rm(ii)]
$d=\frac{y}{x}+\left(\Tr_g\left(\frac{y}{x}\right)+1\right)\frac{x}{\Tr(x)}+\delta,$
\end{itemize}
with $\Tr_g(\delta)=0$ and $\delta\in \K(x)$.
 \end{lemma}
\begin{proof} Since $\Tr_g(c)\neq 0$, we have that $g(c)\neq c$.
{}From $\varphi g= g^{-1} \varphi$,
$$\varphi(\Tr_g(c))=g^{-1}(\Tr_g(\varphi(e)))=g^{-1}\,\Tr_g(g(c))=g^{-1}(\Tr_g(c))=\Tr_g(c).$$
By (iii) of Lemma \ref{nuovissimo},
\begin{eqnarray*}
\varphi(e)+e &=& \frac{1}{\Tr_g(c)} \bigg( \sum_{v=0}^{n-1}a_{g^v}g^v(c)+  \sum_{v=0}^{n-1}(a+a_{g^{-v+1\pmod n}})g^{-v+1}(c) \bigg)\\
&=& \frac{1}{\Tr_g(c)} \sum_{v=0}^{n-1}ag^{-v+1}(c)=a\frac{\Tr_g(c)}{\Tr_g(c)}=a.
\end{eqnarray*}
It has been already noticed that $\Tr_g(x/\Tr_g(x))=1$. Hence, if $d\in \K(x)$ then $\delta=d+(x/\Tr_g(x))$ has zero relative trace. Here $\delta\in \K(x)$ because $x/\Tr_g(x)\in \K(x)$.

To show the last assertion, observe that
\begin{equation}
\label{witteq3}
d^2+d\in \K(x).
\end{equation}
Assume that $\varphi(d)\neq d$. Then  $d\not\in \K(x)$ and $d=\omega_1y+\omega$ with $\omega_1,\omega\in \K(x)$ and $\omega_1\neq 0$. {}From (\ref{witteq3})
$$\omega_1^2y^2+\omega_1y+\omega^2+\omega=\omega_1^2(xy+x^3+\mu)+\omega_1y+\omega^2+\omega=\omega_1(\omega_1x+1)y+\omega_1^2(x^3+\mu)+\omega^2+\omega$$
belongs to $\K(\cX)$, whence $\omega_1=1/x$. Observe that
$$\varphi\left(\Tr_g\left(\frac{y}{x}\right)\right)=\Tr_g\left(\varphi\left( \frac{y}{x}\right)\right)=\Tr_g\left(\frac{x+y}{x}\right)=\Tr_g\left(\frac{y}{x}+1\right)=\Tr_g\left(\frac{y}{x}\right).$$
This shows that  $$\Tr_g\left(\frac{y}{x}\right)\in \K(x).$$
Hence
$$\left(\Tr_g\left(\frac{y}{x}\right)+1\right)\frac{x}{\Tr_g(x)}\in \K(x).$$
Since $\Tr_g(d)=1$ and $$\Tr_g\left(\frac{y}{x}+\left(\Tr_g\left(\frac{y}{x}\right)+1\right)\frac{x}{\Tr_g(x)}\right)=1,$$
the assertion follows.

\end{proof}


\section{Preliminaries to the proof of Theorem \ref{princ1}}\label{sec2}
In this section,  $S$ is a $2$-subgroup of $\aut(\cX)$, that is, a $\K$-automorphism group of $\cX$ whose order is a power of $2$.

The subfield $\K(\cX)^S$ consisting of all elements of $\K(\cX)$
fixed by every element in $S$, also has transcendency degree one
over $\K$. Let $\cY$ be a non-singular model of
$\K(\cX)^S$, that is,
a projective, non-singular, geometrically irreducible, algebraic
curve with function field $\K(\cX)^S$. Sometimes, $\cY$ is called the
quotient curve of $\cX$ by $S$ and denoted by $\cX/S$. The
covering $\cX\mapsto \cY$ has degree $|S|$ and the field extension
$\K(\cX)/\K(\cX)^S$ is  Galois.

Let $\bar{P_1},\ldots,\bar{P_k}$ be the points of the quotient curve $\bar{\cX}=\cX/S$ where the cover $\cX/\bar{\cX}$ ramifies. For $1\leq i\leq k$, let $L_i$ denote the set of points of $\cX$ which lie {over} $\bar{P_i}$. 
In other words, $L_1,\ldots,L_k$ are the short orbits of
$S$ on its faithful action on $\cX$. Here the orbit of $P\in \cX$
$$o(P)=\{Q\mid Q=P^g,\, g\in S\}$$ is {\em long} if $|o(P)|=|S|$, otherwise $o(P)$ is {\em short}. It may be that $S$ has no short orbits. This is the case if and only if every non-trivial element in $S$ is fixed--point-free on $\cX$. On the other side, $S$ has a finite number of short orbits.

If $P$ is a point of $\cX$, the stabilizer $S_P$ of $P$ in $S$ is
the subgroup of $S$ consisting of all elements fixing $P$. For a
non-negative integer $i$, the $i$-th ramification group of $\cX$
at $P$ is denoted by $S_P^{(i)}$ (or $S_i(P)$ as in \cite[Chapter
IV]{serre1979})  and defined to be
$$S_P^{(i)}=\{g\mid \ord_P(g(t)-t)\geq i+1, g\in
S_P\}, $$ where $t$ is a uniformizing element (local parameter) at
$P$. Here $S_P^{(0)}=S_P^{(1)}=S_P$. Furthermore,
for $i\geq 1$, $S_P^{(i)}$ is a normal subgroup of $S_P$ and the
factor group $S_P^{(i)}/S_P^{(i+1)}$ is an elementary abelian
$p$-group. For $i$ big enough, $S_P^{(i)}$ is trivial.

Let $\bar{\gg}$ be the genus of the quotient curve $\bar{\cX}=\cX/S$. The Hurwitz
genus formula  gives the following equation
    \begin{equation}
    \label{eq1}
2\gg-2=|S|(2\bar{\gg}-2)+\sum_{P\in \cX} d_P.
    \end{equation}
    where
\begin{equation}
\label{eq1bis}
d_P= \sum_{i\geq 0}(|S_P^{(i)}|-1).
\end{equation}

Let $\gamma$ be the $2$-rank of $\cX$, see \cite[Section 6.7]{hirschfeld-korchmaros-torres2008}. It is known that $\gamma\leq {\gg}$. If equality holds then $\cX$ is
a {\em{general}} curve, see
\cite[Theorem 6.96]{hirschfeld-korchmaros-torres2008} and \cite{frey1979}.
Let $\bar{\gamma}$ be the $2$-rank of the quotient curve $\bar{\cX}=\cX/S$.
The Deuring-Shafarevich formula, see \cite{sullivan1975} or \cite[Theorem 11,62]{hirschfeld-korchmaros-torres2008}, states that
\begin{equation}
    \label{eq2deuring}
\gamma-1={|S|}(\bar{\gamma}-1)+\sum_{i=1}^k (|S|-\ell_i)
    \end{equation}
where $\ell_1,\ldots,\ell_k$ are the sizes of the short orbits of $S$.

Besides the Hurwitz and the Deuring-Shafarevich formulae which are our main tools from Algebraic geometry, we also need some technical results.
\begin{proposition}
\label{secondram} Assume that $S$ fixes the point $P\in \cX$. Let $i\geq 2$ be the smallest integer
for which the $i^{th}$ ramification group $S_P^{(i)}$ of $S$ at $P$ is trivial. If $S$ has order $2$, then $i$ is even.
\end{proposition}
\begin{proof}
Since $S$ has order two, $\cX$ is a double cover of the quotient curve $\bar{\cX}=\cX/S$. Hence, $\K(\cX)$ is an Artin-Schreier extension of $\K(\cX)^S=\K(\bar{\cX})$. By (c) of Lemma 3.7.8 in \cite{STI2ed}, the different exponent $d_P$ is even. Then the claim follows from \eqref{eq1bis}.
\end{proof}
\begin{proposition}
\label{zeroprank} If $\gamma=0$, then $S$ has a (unique) fixed point.
\end{proposition}
For a proof, see \cite{gktrans}{;} see also \cite[Section 11.15]{hirschfeld-korchmaros-torres2008} and \cite{gklondon}.
\begin{proposition}[Nakajima's bound]
\label{nakajimabound} If $\gamma=1$, then $|S|\leq 4(\gg-1)$ and if $\gamma\geq 2$ then $|S|\leq 4(\gamma-1)$.
\end{proposition}
For a proof, see \cite{nakajima1987}; see also \cite[Theorem 11.84]{hirschfeld-korchmaros-torres2008}.

{}Our main tool from Group theory is Suzuki's characterization of dihedral and semi-dihedral $2$-groups, see \cite[Lemma 4]{suzuki1951}. We stress  that the dihedral group ${\bf{D}}_n$ of order $2n=2^{m+1}$ with $m\geq 3$, as well as the semi-dihedral group ${\bf{SD}}_n$ group of the same order, are generated by an element $g$ of order $2^m$ together with an involution $h$. But the relation linking $g$ and $h$ is $hgh=g^{-1}$ in ${\bf{D}}_n$, while it is $hgh=g^{2^{m-1}-1}$ in ${\bf{SD}}_n$. Another difference between ${\bf{D}}_n$ and $\bf{DS}_n$ is that ${\bf{D}}_n$ contains exactly $n+1$ involutions, namely $g^{2^{m-1}}$ and all $g^ih$, while ${\bf{SD}}_n$ does only $2^{m-1}+1$, namely $g^{2^{m-1}}$ and $g^ih$ with even $i$.
\begin{proposition}[Suzuki's classification]
\label{suzukicl} A $2$-group $H$ which contains an involution whose centralizer has order $4$ is either dihedral, or semi-dihedral, or it has order $4$.
\end{proposition}
We also need a few technical lemmas on finite $2$-groups.
\begin{lemma}{\rm{\cite[Satz 14.9]{huppertI1967}}}.
\label{gr0}
Up to isomorphisms, there exist exactly four non-abelian groups $H$ of order ${2^{m+1}}\geq 16$ containing a cyclic subgroup of index $2$, namely the dihedral group, the semi-dihedral group, the generalized quaternion group, and the group generated by an element $g$ of order $2^{m-1}$ together with an involution $h$ such that $hgh=g^{1+2^{m-2}}$. The number of involutions of $H$ is equal to 
{$2^{m}+1, 2^{m-1}+1,1,3$}
respectively.
\end{lemma}

\begin{lemma}
\label{lemgr} Let $H$ be a transitive permutation group on a set $\Delta$ whose $1$-point stabilizer of {$H$} has order two. 
{Let} $\Delta_w$ be the set of all fixed points of  an involution $w$ {of $H$}. Then $|C_{H}(w)|=2|\Delta_w|.$
\end{lemma}
\begin{proof} It is enough to observe that $g\in H$ leaves $\Delta_w$ invariant if and only $g\in C_{H}(w)$.
\end{proof}
\begin{lemma}
\label{lemgr1} Let $u$ be a central involution of a $2$-group $H$ of order at least $16$. Assume that  $\bar{H}=H/\langle u \rangle$ is a dihedral group. Let $\bar{C}$ {be} a maximal cyclic subgroup of $\bar{H}$. Then the counter-image $C$ of $\bar{C}$ under the natural epimorphism  ${\tau:} H\to \bar{H}$ is either a cyclic subgroup of $H$, or it is a direct product $E\times
\langle u \rangle$ with a cyclic subgroup {$E$}.
\end{lemma}
\begin{proof}
Take an element $c\in H$ such that $\bar{c}=\tau(c)$ is a generator of
$\bar{C}$. Then, either $\langle c \rangle=C$ and $C$ is cyclic,
or $E=\langle c \rangle$ is a cyclic subgroup of $C$ of index $2$.
In the latter case, $u\not \in E$ and hence $C=E\times \langle u \rangle$.
\end{proof}
\begin{lemma}{\rm{\cite[Satz 14.10]{huppertI1967}}}.
\label{group8} Up to isomorphisms, there are five groups of order $8$. Two of them are non-abelian, namely the dihedral and the quaternion groups.
\end{lemma}


\section{Central involutions in $\aut(\cX)$}
We begin with a number of results valid for curves $\cX$ of genus $\gg\geq 2$ which satisfy both hypotheses (I) and (II).
\begin{lemma}
\label{nakaimp} The $2$-rank $\gamma$ of $\cX$ is at least $2$.
\end{lemma}
\begin{proof} {}From Proposition \ref{zeroprank}, $\gamma\geq
1.$ To prove the assertion by absurd, assume that $\gamma=1$.
Let $u\in Z(S)$ be an involution that fixes a point on
$\cX$. From (\ref{eq2deuring}) applied to $U=\langle u \rangle$,  the $2$-rank of the
quotient curve $\bar{\cX}=\cX/U$ is equal to $0$, and $u$ fixes precisely two
points on $\cX$, say $P_1$ and $P_2$. As $u\in Z(S)$, the set $\{P_1,P_2\}$ is preserved by $S$. Therefore, the stabilizer $S_{P_1}$
of $P_1$ in $S$ has index two in $S$, and it
fixes $P_2$ as well. Let $\bar{P_1}$ and $\bar{P_2}$ be the points of $\bar{\cX}$ lying under $P_1$ and $P_2$, respectively. Obviously, $\bar{P_1}\neq \bar{P_2}$. Furthermore,
the factor group $S_{P_1}/U$ is a subgroup of $\aut(\bar{\cX})$, and it fixes both $\bar{P_1}$ and $\bar{P_2}$. Since $\bar{\cX}$ has zero $2$-rank, Proposition \ref{zeroprank} implies that
$S_{P_1}/U$ is trivial. Therefore, $S_{P_1}=U$ and hence $|S|=4$; a contradiction with (I).
\end{proof}
\begin{lemma}
\label{lem00}
$S$ has exactly two short orbits on $\cX$, the larger one of size $\ell_1=\ha\,|S|$ and the shorter one of size $2\leq \ell_2\leq \qa |S|$.
\end{lemma}
\begin{proof}
Let $\bar{\gamma}$ be the $2$-rank of the quotient curve $\bar{\cX}=\cX/S$.
{}From (\ref{eq2deuring}),
$$
\gamma -1=\bar{\gamma}|S|-|S|+\sum_{i=1}^k(|S|-\ell_i)=(\bar{\gamma}+k-1)|S|-\sum_{i=1}^k\ell_i\ge (\bar{\gamma}+\frac{k}{2}-1)|S|,
$$
where $\ell_1,\ldots,\ell_k$ are the sizes of the short orbits of
$S$.

If no such short orbits exist, then $\gamma-1=|S|(\bar{\gamma}-1)$ holds, whence $\bar{\gamma}>1$ follows by $\gamma\geq 2$.
For $\bar{\gamma}>1$, this equation yields that $|S|{\le }(\gamma-1)\leq (\gg-1)$ contradicting (I).

Therefore, $k\geq 1$, and if $\bar{\gamma}\geq 1$ then the above equation implies that $|S|\leq 2(\gamma-1)\leq 2(\gg-1)$, a contradiction with (I).

So, $\bar{\gamma}=0$ and $1\leq k \leq 2$. Actually, $k$ must be $2$, $\gamma{\geq} 2$ being inconsistent with $k=1$ and $\bar{\gamma}=0$ in the above equation.

Therefore, $S$ has precisely two short orbits
say $\Omega_1$ and $\Omega_2$,
and $$\gamma-1=|S|-(\ell_1+\ell_2)$$  with $|\Omega_1|=\ell_1$ and
$|\Omega_2|=\ell_2$.

Assume without loss of generality that
$\ell_1\ge {\ell}_2$. Obviously, $\ell_2<\ha\,|S|$, as otherwise we would have
$\gamma=1$ contradicting Lemma \ref{nakaimp}.
Also, $\ell_1>\qa\,|S|$, since $\gamma-1\ge
|S|(1-\qa-\qa)$ is inconsistent with (I).
Then,
\begin{equation}
\label{elle1}
\ell_1=\ha\,|S|,
\end{equation}
and
\begin{equation}\label{elle2}
 \gamma-1=\ha\,|S|-\ell_2,
\end{equation}
with $\ell_2\le \qa\,|S|$.
\end{proof}
We keep up the notation introduced in the preceding proof. So, $\Omega_1$ and $\Omega_2$ stand for the two short orbits of $S$ on $\cX$. Here  $\ell_1=|\Omega_1|=\ha\,|S|$ while $2\leq \ell_2=|\Omega_2|\leq \qa |S|$. To investigate the smallest case $\ell_2=2$ some technical lemmas are needed.
\begin{lemma}
\label{lem001} If $P\in \Omega_1$ then $|S_P|=2$. If $Q\in \Omega_2$ then
\begin{equation}
\label{eq001}
2\gg-2\geq |S|+\ell_2\big(|S_Q^{(2)}|+|S_Q^{(3)}|-4+{\sum_{i\ge 4} |S_Q^{(i)}|-1}\big),
\end{equation}
{and equality holds if and only if the genus of the quotient curve $\cX/S$ is equal to zero.}
\end{lemma}
\begin{proof}
{The first assertion clearly follows from $\ell_1=\ha\,|S|$. Let $\bar \gg$ be the genus of the quotient curve $\cX/S$.}
{}From (\ref{eq1}) applied to $S$,
$$2\gg-2=(2\bar \gg-2)|S|+\ha\,|S|(2(|S_P|-1)+|S_P^{(2)}|-1+{\ldots})+\ell_2(2(|S_Q|-1)+|S_Q^{(2)}|-1+{\ldots}).$$
This together with  $|S_P|=2$ give
$$2\gg-2=(2\bar \gg+1) |S|+\ha\,|S|(|S_P^{(2)}|-1+{|S_P^{(3)}|-1+\ldots})+\ell_2(-2+|S_Q^{(2)}|-1+\ldots).$$
{If $|S_P^{(2)}|=2$, then by Proposition \ref{secondram} $|S_P^{(3)}|=2$, which contraditcs (I)}.

Therefore, $|S_P^{(2)}|=1$ and hence
(\ref{eq001}) holds.
\end{proof}

\begin{lemma}
\label{lem3} If $u$ is a central involution of $S$ which fixes a point of\,\, $\Omega_2$, then $u$ fixes $\Omega_2$ pointwise but fixes no point outside $\Omega_2$.
\end{lemma}
\begin{proof}
Since $\Omega_2$ is an orbit of $S$ and $u\in Z(S)$, $u$ fixes $\Omega_2$ pointwise.
Assume on the contrary that $u$ also fixes a point on $\Omega_1$. Then $u$ must fix $\Omega_1$ pointwise. From (\ref{eq2deuring}) applied to $U=\langle u \rangle$,
$$\gamma-1=2(\gamma'-1)+\ha|S|+\ell_2,$$
where $\gamma'$ stands for the $2$-rank of the quotient curve $\cX'=\cX/U$. Since $\ell_2\geq 2$, this yields that $\gg-1\geq \gamma-1\geq \ha|S|$ contradicting (I).
\end{proof}

\begin{lemma}
\label{lembis} If a central involution $u$ of $S$ fixes a point of $\Omega_1$ then $u$ fixes $\Omega_1$ pointwise, $\ell_2=2$  and $\cX$ is a hyperelliptic curve.
\end{lemma}
\begin{proof} {}From Lemma \ref{lem3}, $S$ fixes no point outside $\Omega_1$. The argument in the proof of that Lemma applied to $\Omega_1$ proves the first assertion and gives the equation
$$\gamma-1=2(\gamma'-1)+\ha|S|,$$
where $\gamma'$ stands for the $2$-rank of the quotient curve $\cX'=\cX/U$, {with $U=\langle u \rangle$}. This and (\ref{elle2}) imply that $\gamma'=0$ and $\ell_2=2$. In particular, $\cX$ is a hyperelliptic cruve.
\end{proof}

\begin{lemma}
\label{indu1} If $\ell_2>2$ then every non-inductive central involution of $S$ fixes a point on $\cX$.
\end{lemma}
\begin{proof} Let $u$ be a non-inductive central involution of $S$ and assume on the contrary that $u$ fixes no point on $\cX$.  From (\ref{eq1}) applied to $U=\langle u \rangle$,
$${2\gg-2}=2(2\bar{\gg}-2),$$
where $\bar{\gg}$ is the genus of the quotient curve $\bar{\cX}=\cX/U$. Therefore, $\bar{\gg}\geq 2$ and
$$|\bar{S}|=\ha |S|>{\gg}-1=2(\bar{\gg}-1).$$
Furthermore, $\ell_2>2$ yields that $|S|\geq 16$, whence $|\bar{S}|\geq 8$. Since $u$ is non-inductive, $\bar{S}$ must have
a fixed point on $\bar{\cX}$. If $\bar{R}\in \bar{\cX}$ is such a point, and $R_1,R_2\in \cX$ are the points lying over $\bar{R}$, then $S$ leaves the pair $\{R_1,R_2\}$ invariant. Hence, $\Omega_2$ consists of the points $R_1$ and $R_2$. But then $\ell_2=2$, a contradiction.
\end{proof}
\begin{lemma}
\label{indu2} If $S$ has a non-inductive central involution then either $\ell_2=2$, or $\ell_2=\qa |S|\geq 4$. In the latter case, $\cX$ is a general,
bielliptic curve with $|S|=4(\gg-1)$.
\end{lemma}
\begin{proof} 
 Suppose that $\ell_2>2$ and take a non-inductive central involution $u$ of $S$. By Lemmas \ref{indu1}, \ref{lem3} and \ref{lembis}, the set of fixed points $u$ is $\Omega_2$.
{}From (\ref{eq2deuring}) applied to $U=\langle u \rangle$,
\begin{equation}
\label{eq02}
\gamma-1= 2(\bar{\gamma}-1)+\ell_2.
\end{equation}
where $\bar{\gamma}$ is the $2$-rank of the quotient curve $\bar{\cX}=\cX/U$.
Comparing this with (\ref{elle2}) shows that $\bar{\gamma}=0$ is inconsistent with $\ell_2\leq \qa |S|$. So, the case $\bar{\gamma}=0$ does not actually occur.

If $\bar{\gamma}=1$ then (\ref{eq02}) and (\ref{elle2}) give $\ell_2=\qa|S|$. In this case, $\bar{\gg}\geq \bar{\gamma}\geq 1$. {}From (\ref{eq1}) applied to $U=\langle u \rangle$,
$$2\gg-2=2(2\bar{\gg}-2)+\qa\,|S|d_P$$
where $P$ is any point in $\Omega_2$. {}From Proposition \ref{secondram}, either $d_P=2$ or $d_P\geq 4$. The latter cannot actually occur by (I). 
Since the central involution $u$ is non-inductive, one of the cases (A),(B) and (C) occurs. Since $\qa|S|\geq 4$, that is, $|\bar{S}|=\ha|S|\geq 8$, case (B) is ruled out. If case (C) occurred then $S$ would have an orbit of length $2$, contradicting the hypothesis $\ell_2>2$. Therefore, case (A) holds. As $\bar{\gg}\geq \bar{\gamma}=1$, we have that $\bar{\gg}=1$.
This implies that $|S|=4(\gg-1)=4(\gamma-1)$ and hence
$\cX$ is a general curve. 
Therefore, $\cX$ is bielliptic as $u$ is an involution and $\cX/U$ is an elliptic curve.

Let $\bar{\gamma}\geq 2$. This time, (\ref{eq02}) and (\ref{elle2}) give
\begin{equation}
\label{eq211}
\ell_2=\qa |S|-(\bar{\gamma}-1).
\end{equation}
{}From this, $|S|\geq 16$ and hence $|\bar{S}|\geq 8$. Also, $|\bar{S}|=\ha |S|>g-1>2(\bar{\gg}-1)$.
Since $u$ is a non-inductive central involution, $\bar{S}$ has a fixed point on $\bar{\cX}$. But this implies that $\ell_2=2$ as in the final part in the proof of Lemma \ref{indu1}.
\end{proof}

\section{Proof of Theorem \ref{princ1}}
\label{dim1}
We prove two theorems. They together with Lemmas \ref{indu1} and \ref{indu2} provide a proof of Theorem {\ref{princ1}}.
\begin{theorem}
\label{case(II)} Let $\cX$ be a curve of genus $\gg\geq 2$ defined over an algebraically closed field $\K$ of characteristic $2.$ Assume that
$\aut(\cX)$ has a subgroup $S$ of order a power of $2$ satisfying both hypotheses {\rm{(I)}} and {\rm{(II)}}. If $\ell_2=2$ then case {\rm{(ii)}} of Theorem {\rm\ref{princ1}} holds.
\end{theorem}
\begin{proof} Hypothesis {(I)} together with (\ref{elle2}) yield $$|S|>2(\gg-1)\geq 2(\gamma-1)=|S|-4.$$ Since $|S|$ is a power of $2$ bigger than four, two possibilities arise only. Either
\begin{itemize}
\item[(A)] $|S|=2\gg$ and $\gg=\gamma+1$, or
\item[(B)] $|S|=2\gg+2$ and $\gg=\gamma$.
\end{itemize}
{In both case, from \eqref{eq1} applied to $S$ we deduce that the genus of the quotient curve $\cX/S$ is equal to $0$.}

To rule out case (A), suppose on the contrary that $\gg=\ha|S|$. Lemma \ref{lem001} for $\gg=\ha\,|S|$ implies that $|S_Q^{(2)}|=2$ but $|S_Q^{(3)}|=1$, 
{which contradicts Proposition \ref{secondram}}.

In case (B), Lemma \ref{lem001} implies that the second ramification group $S_R^{(2)}$ is trivial at every $R\in \Omega_1\cup \Omega_2$ and hence at every point in $\cX$.  Also, since $\ell_2=2$, the stabilizer $D$ of $Q\in \Omega_2$ in $S$ is an elementary abelian group $D$ of order $\ha|S|$. From (\ref{eq2deuring}) applied to $D$,
$$\gamma-1=|D|(\tilde{\gamma}-1)+|D|-2+\sum_{i=1}^k (|D|-l_i)$$
where $l_1,\ldots,l_k$ are {sizes of} the short orbits {$\Lambda_1,\ldots,\Lambda_k$} of $D$ disjoint from $\Omega_2$. This together with (\ref{elle2}) yield that either $\bar{\gamma}=0$, $k=2$ and {$l_1=l_2=\ha\,|D|$}, or no non-trivial element of $D$ fixes a point of $\cX$ outside $\Omega_2$.

We show that the former case cannot actually occur. The factor group $\bar{S}=S/D$ is a $\K$-automorphism group of the quotient {curve} $\bar{\cX}=\cX/D$. Set $\Omega_2=\{P_1,P_2\}$. Let $\bar{P}_1$ and $\bar{P}_2$ be the points of $\bar{\cX}$ lying under $P_1$ and $P_2$, respectively. Since $D$ fixes both $P_1$ and $P_2$ while $S$ interchanges them, $\bar{S}$ interchanges $\bar{P}_1$ with $\bar{P}_2$. In particular, these points of $\bar{\cX}$ are not fixed by $\bar{S}$. Assume that {$l_1=l_2=\ha|D|$}. Let $\bar{\Lambda}_1, \bar{\Lambda}_2$ be the points of $\bar{\cX}$ under the $D$-orbits {$\Lambda_1$ and $\Lambda_2$}. Since
{$\Omega_1=\Lambda_1\cup \Lambda_2$} and $S$ acts transitively on $\Omega_1$, $\bar{S}$ interchanges $\bar{\Lambda}_1$ with  $\bar{\Lambda}_2$. Since $\Omega_1$ and $\Omega_2$ are the only short orbits of $S$, it turns out that $\bar{S}$ has no fixed point on $\bar{\cX}$. On the other hand, Proposition \ref{zeroprank} shows that $\bar{S}$ must have a fixed point on $\bar{\cX}$, a contradiction.

For a point $P\in \Omega_1$, let $u\in S$ be the unique non-trivial element in $S_P$. Then $u$ is an involution not contained in $D$. Let $U=\langle u \rangle$. Then $S=\langle D, U\rangle$. More precisely, since $D$ and $U$ have trivial intersection, $S=D\rtimes U$. If $S$ is abelian, then $u$ is a central involution, and hence $\cX$ is hyperelliptic by  Lemma \ref{lembis}. This completes the proof.
\end{proof}
\begin{rem}
\label{reml=2}
If $|S|=8$, then $\ell_2=2$ and Lemma \ref{group8} yields that $S$ is either elementary abelian, or dihedral.
\end{rem}

\begin{theorem}
\label{case(I)} Let $\cX$ be a curve of genus $\gg\geq 2$ defined over an algebraically closed field $\K$ of characteristic $2.$ Assume that
$\aut(\cX)$ has a subgroup $S$ of order a power of $2$ satisfying both hypotheses {\rm{(I)}} and {\rm{(II)}}. If $\ell_2=\qa |S|>2$ and some central involution of $S$ fixes a point, then case {\rm{(i)}} of Theorem {\rm\ref{princ1}} holds.
\end{theorem}
\begin{proof}
{}From Lemmas \ref{lem3} and \ref{lembis}, there exists an involution $u\in Z(S)$ which fixes $\Omega_2$ pointwise but no point from $\Omega_1$. Furthermore, $|S|\geq 16$.

Let $W\in \Omega_1$. By (\ref{elle1}) the stabilizer $S_W$ of $W$ in $S$ has order
two. Hence $S_W$ consists of an involution $w$ together with
the identity. Note that $w\neq u$ by Lemma \ref{lembis}. Let $\Omega_w$ be the set of all fixed points of
$w$. Since both $\Omega_1$ and $\Omega_2$ have even size,
$\Omega_w$ also has even size.

If $|\Omega_w|=2$ then $|C_S(w)|=4$ by Lemma \ref{lemgr}, and Proposition \ref{suzukicl} yields that $S$ is either dihedral, or
semi-dihedral. The former case gives (ia). We must show that the latter case cannot actually occur.

Suppose on the contrary that $S\cong {\bf{SD}}_m$ with $m=|S|$. Since $|\Omega_w|=2$, the conjugacy class of $w$ in $S$ consists of $\qa|S|$ involutions. Since $u$ is a further involution of $S$,
{Lemma} \ref{gr0} yields that these $\qa|S|+1$ involutions are all the involutions in $S$. Therefore, the stabilizer {$S_Q$} of any point $Q\in \Omega_2$ has a unique involution, namely $u$.
{}From (\ref{eq1}) applied to $S_Q$,
\begin{equation}
\label{eq16}
2\gg-2\geq |{S_Q}|(2{\tilde{\gg}}-2)+\sum_{P\in \Omega_2} d_P.
\end{equation}
Here {$S_Q$} is a cyclic group of order $4$. Now, ${|S_Q|}=|S_Q^{(1)}|=4$ and {since the factor group $S_Q^{(1)}/S_Q^{(2)}$} is elementary abelian, either $S_Q^{(2)}=S_Q$ or
${S_Q^{(2)}}=\langle u \rangle$. From Proposition \ref{secondram}, {in the latter case} $S_Q^{(3)}=S_Q^{(2)}$ {holds}. Therefore, $d_Q\geq 8$. Since $\Omega_2$ has even size, {$S_Q$} fixes at least one more point $Q'\in \Omega_2$.
By the previous argument, $d_{Q'}\geq 8$. For every other point $P\in \Omega_2$, we have $d_P\geq 4$. From (\ref{eq16}), $2\gg-2\geq |S|$, a contradiction.

Assume now that $|\Omega_w|\geq 4$. Then
$|C_S(w)|\geq 8$, see Lemma \ref{lemgr}.

Let $d={|\Omega_w\cap\Omega_2|}$. Consider the subgroup $M$ of $S$ of
order $4$ generated by $u$ and $w$. Let $\gamma_w$ be the $2$-rank
of the quotient curve $\cX/M$. From (\ref{eq2deuring}) applied to $M$,
\begin{equation}
\label{modeq1} \gamma-1\geq
4(\gamma_w-1)+(|\Omega_w|-d)+3d+(\qa|S|-d)=4(\gamma_w-1)+|\Omega_w|+d+\qa|S|.
\end{equation}
{By \eqref{elle2}, $d=0,\,|\Omega_w|=4$,
$\gamma_w=0$, and equality holds in \eqref{modeq1}. Therefore,}
 the following assertions hold.
\begin{itemize}
    \item[(a)] $M$ has exactly $\eit |S|+2$ short orbits, each of size two; namely two orbits of $\langle u \rangle$ in $\Omega_w$ and each of the orbits of $\langle w \rangle$ in $\Omega_2$.
    \item[(b)] $uw$ has no fixed point on $\Omega_1$.
    \item[(c)] $\Omega_w\subseteq \Omega_1$ with
\begin{equation}
\label{nagabiseq3} |\Omega_w|=4.
\end{equation}
\end{itemize}
Since $|S|\geq 16$, from (\ref{nagabiseq3}) it follows that ${|\Omega_w|}<\ha\,|S|={|\Omega_1|}$. This together with (b) imply that $S$ has at least five involutions.

By Lemma \ref{lemgr},
\begin{equation}
\label{nagabiseq31} |C_S(w)|=8.
\end{equation}
Since $|S|\geq 16$, this yields that $S$ is not abelian.

Let $\tau$ be the natural group homomorphism from $S\to \bar{S}$ where $\bar{S}$ is the factor group $S/\langle u \rangle$. Note that $\bar{S}$ is a $\K$-automorphism group of {the quotient curve}
$\bar{\cX}{=\cX/S}$ of order at least eight, and we are going to show that $\bar{S}$ is  either a dihedral or a semi-dihedral group.

By Lemma \ref{lem3}, $u$ fixes no point on $\Omega_1$. Therefore, $|\bar{\Omega}_1|=\ha|\Omega_1|$ where the set $\bar{\Omega}_1$ consists of all points of $\bar{\cX}$ lying under the points of $\Omega_1$ with respect to the {covering $\cX\to\bar{\cX}$}. Also, $\bar{S}$ is a transitive permutation group on $\bar{\Omega}_1$. Take two points, $P,R\in \Omega_w$ such that $R\neq u(P)$. Then
\begin{equation}
\label{eq04}
\Omega_w=\{P,u(P),R,u(R)\}.
\end{equation}
Let $\bar{P}$ and $\bar{R}$ be the points of $\bar{\cX}$ lying under $P$ and $R$, respectively.
Then $\bar{P},\bar{R}$ are the only fixed points of $\bar{w}=\tau(w)$ on $\bar{\Omega}_1$. From Lemma \ref{lemgr},
$|C_{\bar{S}}(\bar{w})|=4$. Proposition \ref{suzukicl} yields that $\bar{S}$ is either dihedral, or
semi-dihedral. These two
possibilities are investigated separately.

Assume that $\bar{S}$ is dihedral. Let $\bar{C}$ be a (maximal) cyclic subgroup of $\bar{S}$ of order $\qa |S|$. Set $C=\tau^{-1}(\bar{C})$. From Lemma \ref{lemgr1}, either $C$ itself cyclic, or $C=E\times \langle u \rangle$ with a cyclic subgroup $E$.

If $w\in C$, then $u\in C$ implies that $C$ has at least two involutions. Hence $C=E\times \langle u \rangle$. Furthermore, the only involution in $E$  is either $w$ or $uw$.
{}From Lemma \ref{lembis} and assertion (b), neither $u$ nor $uw$ has a fixed point on $\Omega_1$. Suppose that $S$ has an involution $w'$, $w'\neq w$, with a fixed point in $\Omega_1$. Since $S$ is transitive on $\Omega_1$ and the $1$-point stabilizer of $S$ on $\Omega_1$ has order two, we have that $w$ and $w'$   are conjugate under $S$. Since $C$ is a (normal) subgroup of $S$ of index $2$ and $w\in C$, this implies that $w'$ is also in $C$. But then we would have either $w'=u$ or $w'=u${$w$}, a contradiction. Therefore, every point in $\Omega_1$ must be fixed by $w$. Hence $\Omega_w=\Omega_1$. {}From (\ref{nagabiseq3}), $|\Omega_1|=4$ and hence $|S|=8$, a contradiction.

If $w\not \in C$, then no non-trivial element in $C$ fixes a point
in $\Omega_1$, and hence $C$ is sharply transitive on $\Omega_1$.
Bearing (\ref{eq04}) in mind, take $h\in C$
such that $h(P)=R$. Then $h\neq u$ and $hwh^{-1}(R)=R$. Since the stabilizer of $R$ in $S$ is generated by $w$, this yields that $h\in C_S(w)$ with $h\neq w$.
{Moreover,
$h\in Z(S)$ as $S$ is generated by an abelian group $C$ containing $h$ together with $w$. As $h\neq u$, the center $Z(S)$ contains at least two non-trivial elements, whence $S$ can be neither dihedral or semidihedral. By Lemma \ref{gr0},  $C$ is not cyclic, and therefore $C=E\times \langle u \rangle$ holds.}
Since $h\in Z(S)$,  $h$ preserves $\Omega_w$, and
$$h(u(P))=(hu)(P)=(uh)(P)=u(R).$$
Therefore, the permutation induced by $h$ on $\Omega_w$ is either the product of the transpositions $(PR)$ and $(u(P)u(R))$, or it is the $4$-cycle
$(PRu(P)u(R))$. In the latter case, $h^2=u$ as $C$ is sharply transitive on $\Omega_1$. Actually this is  impossible, because the square of every element of $E\times \langle u \rangle$ of order $\geq 4$ is in $E$, and hence distinct from $u$. Therefore,  $h$ is an involution distinct from $u$. 
Suppose that $h$ fixes a point on $\cX$. Since $h\in Z(S)$ and $h$ does not fix $P$, $h$ has no fixed point on $\Omega_1$. Therefore, $h$ fixes
a point in $\Omega_2$, and hence every point in $\Omega_2$ is fixed by $h$.  Let $L=\langle h,u \rangle$. If
$\tilde{\gamma}$ is the $2$-rank of the quotient curve $\tilde{\cX}=\cX/L$, from (\ref{elle2}) and (\ref{eq2deuring}) applied to $L$,
$$\qa |S|=\gamma-1\geq 4(\tilde{\gamma}-1)+\tqa |S|,$$
whence $|S|\leq 8$, a contradiction. Therefore $h$ is fix{ed}--point--free on $\cX$. From Lemma \ref{indu1}, $h$ is an inductive central involution of $S$.

{Note that $u,h$ and $uh$ are the only three involutions in $C$, and each such involution is central in $S$. As $\bar S$ is dihedral, any other involution in $S$ is not central.}
We show that {$uh$} is fix{ed}--point--free on $\cX$, as well. Suppose on the contrary that $P\in \Omega_2$ is fixed by {$uh$}.
Since {$uh$}$\in Z(C)$, the orbit $\Delta$ of $P$ under $C$ is pointwise fixed by {$uh$}. We have that $|C_P|\leq 4$, as $C_P$ is a subgroup of $S_P$ and $|S_P|=4$. Actually, $|C_P|=4$ since $u,${$uh$}$\in C_P$ and {$uh$} $\neq u$. Hence, $C_P=\{1,u,h,uh\}$, and  $|\Delta|=\eit\,|S|$. Now, choose $Q\in \cX$ from $\Omega_2\setminus \Delta$, and $s\in S$ such that $s$ takes $P$ to $Q$. Then $suhs^{-1}$ fixes $Q$. Since $uh\in C$ and $C$ is a normal subgroup of $S$, this implies that $suhs^{-1}\in C$. Hence, either $suhs^{-1}=h$ or $suhs^{-1}=uh$. In both cases, $C_P=C_Q$.
{}From (\ref{eq2deuring}) applied to $C_P$,
$$\qa\,|S|=\gamma-1=4(\widehat{\gamma}-1)+\teit\,|S|+\teit\,|S|,$$
where $\widehat{\gamma}$ is the $2$-rank of the quotient curve $\widehat{\cX}=\cX/C_P$. But this is only possible for $|S|=8$, a contradiction.

{}From Lemma \ref{indu1}, not only $h$ but also {$uh$} is an inductive central involution. On the other hand, $u$, the third central involution of $S$, is not inductive. In fact, from (\ref{eq1}) applied to $U=\langle u \rangle $ it follows that the genus of the quotient curve $\cX/U$ is equal to $1$. This gives case (ib).

To rule out the case that $\bar{S}$ is semi-dihedral, we give a
lower bound for the number $n_4$ of subgroups of $S$ of order $4$
which contains $u$.

By (\ref{elle1}) and (\ref{nagabiseq3}), $S$
has $\eit|S|$ pairwise distinct subgroups $M=\{1,w,u,uw\}$ when $w$ ranges over the
involutions in $S$ fixing a point of $\Omega_1$.

Since $\ell_2=\qa |S|$ and $u$ fixes $\Omega_2$ pointwise,  the stabilizer
$S_Q$ with $Q\in \Omega_2$ contains $u$ and has order $4$. Let $r$ be the number of
fixed points of $S_Q$ in $\Omega_2$. Obviously $r\geq 1$.
Let $\tilde{\gamma}$ the $2$-rank of the quotient curve $\tilde{\cX}=\cX/S_Q$. {}From (\ref{eq2deuring}) applied to $S_Q$,
$$\gamma-1\geq 4(\tilde{\gamma}-1)+3r+(\qa|S|-r)=4(\tilde{\gamma}-1)+\qa|S|+2r.$$
Since (I) holds, (\ref{elle2}) and $r\geq 1$  yield that $r=2$ and $\tilde{\gamma}=0$.

Since $|S|\geq 16$, this shows that
there is point $R\in\Omega_2$ such that $S_Q\neq S_R$. Therefore,
$$n_4\geq \eit|S|+2.$$
As a consequence, $\bar{S}$ more than $\eit|S|+1=\qa|\bar{S}|+1$ pairwise distinct
involutions. By Proposition \ref{gr0}, $\bar{S}$ is not a semi-dihedral group.
\end{proof}
\section{Some explicit examples}
\label{expex}
In this section, $\K$ is the algebraic closure of the finite field ${\mathbb{F}}_q$ of order $q$ where $q\geq 4$ is a power of $2$, and $w$ a primitive element of ${\mathbb{F}}_q$.
We exhibit several curves with explicit equations that realize the cases in Theorem \ref{princ1}.
\subsection{Case (ia)} In Section \ref{biellipex}, an infinite family of curves $\cX_k$ of type (ia) is constructed. Here we single out the case of $\gg=9$,  and illustrate some computational results for $q=16$. Let $\bar{\cX}$ be the elliptic curve of equation $Y^2+XY+X^3+{\mu}=0,$ and $K(\cX)=K(x,y)$ with $y^2+xy+x^3+{\mu}=0$ is its function field.

In the first construction, take {$\mu$ a primitive element in $\mathbb F_{16}$, and} $k=1$ in (\ref{eq18marzo}). Then the definition (\ref{eq22marzo}) reads
$e_1=(\delta/\xi)y+(\omega/\xi)$ with
$$\delta=\mu x^{13} + \mu^2x^{11} + \mu^{11}x^9 + \mu^{13}x^7 + \mu^{13}x^5 + \mu^5x^3
    + \mu^{11}x,$$
$$\xi=x^{16} + \mu^4x^{12} + \mu x^8+ \mu^6x^4 + \mu^4,$$ and
$$
\begin{array}{lll}
\omega&=&\mu^4x^{16} + \mu x^{15} + \mu^{11}x^{14} + \mu^2x^{13} + \mu^7x^{12} + \mu^{11}x^{11}
    + \mu^5x^{10} + \mu^{13}x^9 + \\ &&\mu^{14}x^8 + \mu^{13}x^7 + \mu^{12}x^6 + \mu^5x^5
    + \mu^3x^4 + \mu^{11}x^3 + \mu^{14}x^2 + \mu^8.
\end{array}
    $$
Let $\cX$ be a non-singular model of the bielliptic function field which is the extension of $K(\bar{\cX})$ by adjoining $z$ where $z^2+z+e_1=0$.
Eliminating $y$ from $z^2+z+e_1=0$ and $y^2+xy+x^3+\mu=0$, gives an affine equation of a plane (singular) model of $\cX$:
\begin{eqnarray*}
F(X,Z)&=& Z^4X^{28} + \mu Z^4X^{26} + \mu^7Z^4X^{24} + \mu^3Z^4X^{22} + \mu^8Z^4X^{20} +
\\ & & \mu^7Z^4X^{18} + \mu^4Z^4X^{16} + \mu^8Z^4X^{14} + \mu^6Z^4X^{12} +\mu^{13}Z^4X^{10} + Z^4X^8 + 
\\ & & \mu^8Z^4X^6 + \mu^9Z^4X^4 + Z^4X^2+\mu^{11}Z^4 + Z^2X^{28} + \mu^7Z^2X^{24} + \mu^13Z^2X^{22} +
\\ & & \mu^{11}Z^2X^{20} + \mu^{12}Z^2X^{16} + \mu^4Z^2X^{14} + \mu^{11}Z^2X^{12} +\mu^{10}Z^2X^{10} + \mu^3Z^2X^8 + 
\\ & & \mu^8Z^2X^6 + \mu^{11}Z^2X^4 +        \mu^{14}Z^2X^2 + \mu^{11}Z^2 + \mu ZX^{26} + \mu^8ZX^{22} + \mu^7ZX^{20} +
\\ & &  \mu^7ZX^{18} + \mu^6ZX^{16} + \mu^5ZX^{14} + \mu ZX^{12} + \mu^9ZX^{10} +   \mu^{14}ZX^8 + \mu^2ZX^4 + 
\\ & & \mu^3ZX^2 +\mu^8X^{28} + \mu^3X^{26} +  \mu^{13}X^{24} + \mu^9X^{22} + \mu^{13}X^{18} + \mu^8X^{16} + 
\\ & & \mu^5X^{14} +\mu^5X^{12} + \mu^{12}X^{10} + \mu^{10}X^8 + X^6 + \mu^5X^4 + \mu^4=0.
\end{eqnarray*}
Obviously, $\cX$ is defined over ${\mathbb{F}}_{16}$. According to the results in Section \ref{biellipex}, it has genus and $2$-rank equal to $9$ and a dihedral $\K$-automorphism group of order $32$. Therefore $\cX$ is an example of case (ia) of Theorem \ref{princ1}.

For the second construction we relay on Lemma \ref{lemwitt}. In (\ref{eq18marzo}), let $k=7$,  and replace $d=x/{\rm{Tr}}_g(x)$ with $d$ as in (ii) of Lemma \ref{lemwitt}. 
Then $e=(\delta/\xi)y+(\omega/\epsilon)$ with
$$\delta=x^{20} + \mu^4x^{16} + \mu^2x^{14} + \mu^9x^{12} + \mu^{13}x^{10} + \mu^8x^8 +
    \mu^5x^6 + \mu^5x^4 + \mu^5,$$
$$
\begin{array}{lll}
&& \xi=x^{19} + \mu^4x^{15} + \mu x^{11} + \mu^6x^7 +
    \mu^4x^3,
\end{array}
$$
\begin{eqnarray*}
\omega&=&x^{39} + \mu x^{38} + \mu^5x^{37} + \mu^6x^{36} + \mu^{14}x^{35} +
    \mu^9x^{34} + \mu^8x^{33} + \mu^2x^{32} + \mu^{10}x^{31} \\&& +\mu^{10}x^{30} +
    \mu^8x^{29} + \mu^{11}x^{28} + \mu^4x^{27} + \mu^9x^{26} + \mu^3x^{25} +
    \mu^{10}x^{24} + x^{23} + \mu^9x^{22} \\&& + x^{21} + \mu^{11}x^{20} + x^19 + \mu^5x^{18} +
    \mu^6x^{17} + x^{16} + \mu^8x^{14} + \mu^{12}x^{13} + \mu^{10}x^{12} \\&& + \mu^2x^{11}+
    \mu^{13}x^{10} + \mu^5x^9 + \mu^4x^8 + \mu^2x^7 + \mu^8x^6 + \mu^{11}x^5 +
    \mu^6x^4 + \mu^6x^3 \\&& + \mu^7x^2 + \mu^2x + \mu^3;
\end{eqnarray*}
\begin{eqnarray*}
\epsilon&=&x^{36} + \mu^5x^{34} +
    \mu^{14}x^{32} + x^{30} + \mu^2x^{28} + \mu^4x^{26} + \mu^{10}x^{24} + \mu^7x^{22} +
    \mu^4x^{20} + \\ && \mu^{12}x^{18} + \mu x^{16} + \mu^5x^{14} + x^{12} + \mu^{14}x^{10} +
    \mu^{13}x^8 + \mu^{13}x^6 + \mu^5x^4 + \mu x^2.
    \end{eqnarray*}

    This time, we obtain an irreducible plane curve $\cC$ with affine equation
\begin{eqnarray*}
F(X,Y)&=&X^{40} + \mu X^{39}+ X^{38}Y+ \mu^{5}X^{38}+ X^{37}Y^{2}+ X^{37}Y+ \mu^{6}X^{37}+
\mu^{5}X^{36}Y + \\&& \mu^{14}X^{36 }+\mu^{5}X^{35}Y^{2 }+ \mu^{5}X^{35}Y+ \mu^{9}X^{35 }+
\mu^{14}X^{34}Y + \mu^{8}X^{34 }+ \mu^{14}X^{33}Y^{2 }+ \\ &&\mu^{14}X^{33}Y + \mu^{2}X^{33 }+
\mu^{8}X^{32}Y+ \mu^{10}X^{32 }+ X^{31}Y^{2 }+ X^{31}Y+ \mu^{10}X^{31 }+ \mu^{10}X^{30}Y
+ \\&&
\mu^{8}X^{30 }+ \mu^{2}X^{29}Y^{2}+ \mu^{2}X^{29}Y+ \mu^{11}X^{29}+ \mu^{8}X^{28}Y +
 \mu^{4}X^{28 }+ \mu^{4}X^{27}Y^{2 }+ \\&&
 \mu^{4}X^{27}Y+ \mu^{9}X^{27 }+ \mu^{4}X^{26}Y +
    \mu^{3}X^{26 }+ \mu^{10}X^{25}Y^{2 }+ \mu^{10}X^{25}Y+ \mu^{10}X^{25 }+\\&&
  \mu^{3}X^{24}Y+ X^{24 }+ \mu^{7}X^{23}Y^{2 }+ \mu^{7}X^{23}Y+ \mu^{9}X^{23 }+ X^{22}Y+ X^{22 }+
    \mu^{4}X^{21}Y^{2 }+ \\ &&\mu^{4}X^{21}{Y }+ \mu^{11}X^{21 }+ X^{20}Y+ X^{20 }+
    \mu^{12}X^{19}Y^{2 }+ \mu^{12}X^{19}Y + \mu^{5}X^{19 }+ X^{18}Y + \\ &&\mu^{6}X^{18 }+
    \mu X^{17}Y^{2}+ \mu X^{17}Y+ X^{17 }+ \mu^{6}X^{16}Y+ \mu^{5}X^{15}Y^{2 }+
    \mu^{5}X^{15}Y + \mu^{8}X^{15 }+ \\&& \mu^{12}X^{14 }+ X^{13}Y^{2 }+ X^{13}Y + \mu^{10}X^{13 }+
    \mu^{12}X^{12}{Y }+ \mu^{2}X^{12 }+ \mu^{14}X^{11}Y^{2 }+ \mu^{14}X^{11}{Y }+ \\&&\mu^{13}X^{11 }+
    \mu^{2}X^{10}Y + \mu^{5}X^{10 }+ \mu^{13}X^{9}Y^{2 }+ \mu^{13}X^{9}Y + \mu^{4}X^{9 }+
    \mu^{5}X^{8}Y + \mu^{2}X^{8 }+ \\&&\mu^{13}X^{7}Y^{2 }+ \mu^{13}X^{7}Y + \mu^{8}X^{7 }+
    \mu^{2}X^{6}Y + \mu^{11}X^{6 }+ \mu^{5}X^{5}Y^{2 }+ \mu^{5}X^{5}Y + \mu^{6}X^{5 }+\\&&
    \mu^{11}X^{4}Y + \mu^{6}X^{4 }+ \mu X^{3}Y^{2 }+ \mu X^{3}Y + \mu^{7}X^{3 }+ \mu^{6}X^{2}Y
    + \mu^{2}X^{2 }+ \mu^{3}X + \mu^{2}Y{=0}.
 \end{eqnarray*}
A non-singular model $\cX$ has genus and $2$-rank equal to $9$, and it provides another example
of case (ia) of Theorem \ref{princ1}.
\subsection{Case (ib)} Let $q=16$. For a primitive element $\mu$ of ${\mathbb{F}}_{16}$, let $\cX$ be the curve which is the non-singular model of the irreducible plane curve $\cC$ with affine equation  $$
\begin{array}{lll}
F(X,Y)&=&Y^4X^7 + \mu^5Y^4X^5 + \mu^{13}Y^4X^3 + \mu^9Y^4X + YX^7 \mu^5YX^5 + \\
&&\mu^{13}YX^3 + \mu^9YX + X^8 + \mu^2X^6 + \mu^8X^4
        + \mu^3X^2 + \mu^2=0.
\end{array}
$$
{}From a computer aided computation performed by MAGMA, $\cX$ has genus $9$ and $\aut(\cX)$ has a subgroup  $S$ {of} order $32$ such that
${S}\cong D_8\times C_2$. Furthermore, $\bar{\cX}=\cX/C_2$ has genus $5$ and $\aut(\bar{\cX})$ has a dihedral subgroup of order $8$. Therefore,
$\bar{\cX}$ is a curve of type {(ib)}.
\subsection{Case (ii)}
 Let $\cX$ be the hyperelliptic curve which is the non-singular model of the projective irreducible plane curve {$\cC$} of degree $q+2$ with affine equation
\begin{equation*}
\label{exeq} (Y^2+Y+X)(X^q+X)+ \sum_{\alpha\in {\mathbb{F}_q}} \frac{X^q+X}{X+\alpha}\,=0.
\end{equation*}

It is easily seen that $\cC$ has exactly two points at infinity, namely $X_\infty=(1,0,0)$ and $Y_\infty=(0,1,0)$. Both are  ordinary singularities. More precisely, $X_\infty$ and $Y_\infty$ are singular points of $\cC$ with multiplicity $q$ and $2$, respectively. No affine point of $\cC$ is singular. Therefore, $\cX$ has genus
$$\gg=\ha\,(q+1)q-1-\ha\, q(q-1)=q-1,$$ see \cite[Theorem 5.57]{hirschfeld-korchmaros-torres2008}.
For $\beta \in {\mathbb{F}}_q$, let $\mu\in \K$ such that $\mu^2+\mu=\beta$. Then the map $$\varphi_\mu:\,\, (x,y) \to (x+\beta,y+\mu)$$  preserves $\cC$ and hence it is $\K$-automorphism of $\cX$. These maps form a $\K$-automorphism group $S$ of $\cX$. Obviously, $S$ is an elementary abelian group of order $2q$.

Since $2q=2\gg+2$, $\cX$ provides an example for case (ii) of Theorem \ref{princ1}.

\subsection{Case (iii)}\label{caseiii}
Let $\cX$ be the non-singular model of the projective irreducible plane curve $\cC$ of degree $2q$ with affine equation
\begin{equation*}
\label{exeq1} (Y^q+Y)(X^q+X)+1=0.
\end{equation*}

As in the preceding example, $\cC$ has exactly two points at infinity, namely $X_\infty=(1,0,0)$ and $Y_\infty=(0,1,0)$; both are {ordinary singularities} of multiplicity $q$. The tangents to $\cC$ at $X_\infty$ are  the lines $v_\mu$ {with} equation $Y-\mu=0$ with $\mu\in {\mathbb{F}}_q$. Similarly for $Y_\infty$ and the lines $h_\mu$ of equation $X-\mu=0$.
No affine point of $\cC$ is singular.  Therefore $\cX$ has genus
$$\gg=\ha\,(2q-1)(2q-2)-q(q-1)=(q-1)^2,$$
see \cite[Theorem 5.57]{hirschfeld-korchmaros-torres2008}.  For $\ga,\gb\in {\mathbb{F}}_q$ the map
$$\varphi_{\ga,\gb}\,:\,(X,Y)\to (X+\ga,Y+\gb)$$ preserves $\cC$ and so it is a $\K$-automorphism of $\cX$. Here,
$E=\{\varphi_{\ga,\gb}|\ga,\gb\in {\mathbb{F}}_q\}$ is an elementary abelian group of order $q^2$. Also, the map
$$\rho\,:\,(X,Y)\to (Y,X)$$ preserves $\cC$ and hence it is a further $\K$-automorphism of $\cX$. The group generated by $E$ together with $\rho$ is the the semidirect product $E\rtimes \langle\rho\rangle$ and it has order $2q^2$. Since $2q^2>2((q-1)^2-1)=2(\gg-1)$, Nakajima's bound implies that $E\rtimes \langle\rho\rangle$ is not properly contained in a $2$-subgroup of $\aut(\cX)$.
Let $S=E\rtimes \langle\rho\rangle$. It is easily seen that the central involutions of $S$ are the maps $\varphi_{\ga,\ga}$ with $\ga\in {\mathbb{F}}_q$ and $\ga\neq 0$.

We show that no non-trivial element in $S$ fixes a point of $\cX$. Obviously, no non-trivial element in $S$ fixes an affine point. Since the point $U=(1,1,0)$ is not in $\cC$ and $\rho$ interchanges the points $X_\infty$ and $Y_\infty$, no point in $\cX$ is fixed by an element in the coset of $E$ containing $\rho$. This holds true for any non-trivial element in $E$, since $\varphi_{\ga,\gb}$ preserves no line of type $h_\mu$ or $v_\mu$, and hence it preserves no branch centered either at $X_\infty$ or $Y_{\infty}$.

Therefore, every central involution of $S$ is inductive, and hence $\cX$ is an example for case (iii) in Theorem \ref{princ1} with
\begin{equation}
\label{eq18marter}
{\mbox{$|S|=2(\gg-1)+4q-2$ with $\gg=(q-1)^2$ and $q=2^h\geq 4$}}.
\end{equation}
Here, Nakajima's bound is only attained for $q=4$.

\subsection{Example of an inductive sequence of curves}
The procedure described in Introduction starting with $\cX$ {as in Subsection \ref{caseiii}} and ending with a curve free from inductive central involutions is now illustrated in the smallest case, $q=4$. With the above notation, $\gg=9$ and $|S|=4(\gg-1)=32$. As we have pointed out, $u=\varphi_{1,1}$ is an inductive central involution of $S$.
{}From (\ref{eq1}) applied to $\langle u \rangle$,
$$16=2\gg-2=2(2\bar{\gg}-2),$$
where $\bar{\gg}$ is the genus of the quotient curve $\bar{\cX}=\cX/\langle u \rangle$. Hence $\bar{\gg}=5$. Similarly,
$\bar{\cX}$ has $2$-rank $5$. The factor group $\bar{S}=S/\langle u \rangle$ is a subgroup of $\bar{\cX}$ of order $16$.
Thus $|\bar{S}|=16=4(\bar{\gg}-1)$. So, Nakajima's bound is attained by $\bar{\cX}$.
Since the function field $\K(\cX)$ is $\K(x,y)$ with $(x^4+x)(y^4+y)+1=0$, its subfield generated by $t=x+y$ and $z=y^2+y$ is the function field $\K(\bar{\cX})$. It is easily seen that
$(z^2+z)(t^4+t+z^2+z)+1=0$, that is, $\bar{\cX}$ is the {non-singular} model of the projective irreducible plane curve $\bar{\cC}$ with equation $$(X^2+XZ)(Y^4+YZ^3+X^2Z^2+XZ^3)+Z^4=0.$$
{}From computations performed by MAGMA, $\bar{\cX}$ has exactly $28$ $\mathbb{F}_{16}$-rational points.
 Since $\bar{\cX}$ has genus $5$, Nakajima's bound yields that $|\bar{S}|\leq 16$. Actually, the bound is attained as MAGMA computations show that $\aut(\bar{\cX})$ contains the following three $\K$-automorphisms, where $\mu$ is a primitive element of $\mathbb{F}_{16}$:
\begin{eqnarray*}
\psi_1 &=& \,(X,Y,Z) \to (XY^2 + X^2Z + XYZ+ \mu^{10}Y^2Z + XZ^2+ \mu^5YZ^2 + \mu^5Z^3,\\
&& XY^2 + X^2Z + XYZ + \mu^{10}Y^2Z + \mu^{10}YZ^2+ \mu^5Z^3,
Y^2Z + YZ^2 + Z^3);\\
\psi_2 & =& \,(X,Y,Z) \to (X,Y+Z,Z);\\
\psi_3 &=& \,(X,Y,Z) \to (X+Z,Y+Z,Z).
\end{eqnarray*}
They generate indeed a subgroup $\bar{S}$ of order $16$. More precisely, $\langle \psi_1,\psi_2\rangle$ is a dihedral group $D_4$ of order $8$ and $\psi_3$ generates a cyclic group $C_2$ of order $2$ so that $\bar{S}=D_4\times C_2$. The central involutions in $\bar{S}$ are three, namely $\psi_3$, $$\psi_4 = (X,Y,Z)\to(Y^2+XZ+YZ+Z^2,YZ+Z^2,Z^2)$$
and $$\psi_5=(X,Y,Z)\to(Y^2+XZ+YZ,YZ,Z^2).$$
Neither $\psi_3$ nor $\psi_4$ have fixed point on $\cX$ while $\psi_5$ does have four, namely
$$
{P_1=(\mu^5,1,1),\,P_2=(\mu^{10},1,1),\,P_3=(\mu,0,1),\,P_4=(\mu^{10},0,1).}$$
Furthermore, $\bar{S}$ has two orbits on the set of $\mathbb{F}_{16}$-rational points of $\bar{\cX}$, of sizes  $\ell_1=8$ and $\ell_2=4$.
{}From Lemma \ref{indu1}, both $\psi_3$ and $\psi_4$ are inductive involutions of $\bar{S}$.

The quotient curve $\bar{\bar{\cX}}_3=\bar{\cX}/\langle \psi_5\rangle$ is an elliptic curve. This follows from (\ref{eq2deuring}) applied to $\bar{\cX}$ and its $\K$-automorphism group $\langle \psi_5 \rangle$.

Therefore, the central involution $\psi_5$ of $\bar{S}$ is not inductive, and $\bar{\cX}$ provides an example for case (ib) of Theorem \ref{princ1}.

The quotient curve $\bar{\bar{\cX}}_1=\bar{\cX}/\langle \psi_3\rangle$ has genus and $2$-rank $3$, and  equation
$$X^4 + X^2Y^2+Y^4+ X^2YZ+ XY^2Z +X^2Z^2+ XYZ^2+YZ^3=0.$$  Hence $\bar{\bar{\cX}}_1$ is a non-singular plane quartic. Also, $\bar{\bar{S}}_1=\bar{S}/\langle \psi_3\rangle$ is a dihedral group of order $8$. This shows that Nakajima's bound is attained by $\bar{\bar{\cX}}_1$. As we have already pointed out, $\psi_3$ is an inductive central involution of $\bar{S}$ as it fixes no point of $\cX$.
This can also be shown using the fact that $\aut(\bar{\bar{\cX}}_1)$ is the projective group $PSL(2,7)$ and that a dihedral subgroup of $PSL(2,7)$ of order $8$ is known to fix no point in the plane. Therefore $\bar{\cX}$ is an example for case (iii) in Theorem \ref{princ1}, and also illustrates Remark \ref{reml=2} with a dihedral group.

The quotient curve $\bar{\bar{\cX}}_2=\bar{\cX}/\langle \psi_4\rangle$ is a hyperelliptic curve of genus $3$ and $2$-rank $3$, defined by the   affine equation
\begin{eqnarray*}
Y^2+(\mu^{10}X^4 + X^3 + 1)Y = \mu^{13}X^8 + \mu^5X^7 + \mu^3X^6+\\
 + \mu^3X^5 + \mu^{14}X^4 + \mu^7X^3 + \mu^{11}X^2 + X+ 1,
\end{eqnarray*}
and $\bar{\bar{S}}_2=\bar{S}/\langle \psi_4\rangle$ is an elementary abelian group of order $8$. As we have already observed, {$\psi_4$}
is an inductive  central involution. This can also be shown ruling out the possibility that  $\bar{\bar{S}}_2$ fixes a point of $\bar{\bar{\cX}}_1$. For this purpose, assume on the contrary the existence of a point $P\in \bar{\bar{\cX}}_2$ fixed by $\bar{\bar{S}}_2$.
{We} show that there exists another fixed point $P'\in \bar{\bar{\cX}}_1$ of $\bar{\bar{S}}_2$. Observe that $\bar{\bar{\cX}}_2$ is defined over $\mathbb{F}_{16}$. Furthermore, it has exactly $30$ $\mathbb{F}_{16}$-rational points. So, if $P$ is an $\mathbb{F}_{16}$-rational point, $\bar{\bar{S}}_2$ induces a permutation group on the set of the remaining $29$ $\mathbb{F}_{16}$-rational points. As $29$ is an odd number, $\bar{\bar{S}}_2$ must fix some of those points, and $P'$ may be any of them. If $P$ is not defined over $\mathbb{F}_{16}$, the Frobenius image of $P$ can be taken for $P'$. Now, (\ref{eq2deuring}) applied to $\bar{\bar{S}}_2$ gives $2\ge 8(-1)+14$,  a contradiction. Therefore $\bar{\cX}_2$ is another example for case (iii) in Theorem \ref{princ1}, and also illustrates Remark \ref{reml=2} with an elementary abelian group.

\subsection{Example of a curve of genus $\gg$ with a semidihedral $\K$-automorphism group of order $2(\gg-1)$}
For a primitive element $\mu$ of $\mathbb{F}_{16}$, let $\cX$ be a non-singular model of the irreducible plane curve defined with an affine
equation $F(X,Y)=f_1(X)Y^4+f_2(X)Y^2+f_3(X)Y+f_4(X)$ where
$$
\begin{array}{lll}
f_1(X)&=&X^{70}+\mu^{14}X^{66}+\mu^9X^{62}+\mu^{10}X^{58}+\mu^{12}X^{54}+
\mu^5X{46}+\mu^7X^{42}+\mu^{13}X^{38}+\\&&\mu^{2}X^{30}+\mu^{9}X^{26}+
\mu^{10}X^{22}+X^{18}+\mu^{11}X^{10}+\mu^6X^6;
\end{array}
$$
$$
\begin{array}{lll}
f_2(X)&=& X^{72}+X^{70}+\mu^{14}X^{68}+\mu^{13}X^{66}+\mu X^{64}+\mu^{14}X^{62}+
X^{60}+\mu^{13}X^{58}+\\&&X^{56}+\mu^5X^{54}+\mu X^{52}+X^{50}+\mu^5X^{48}+
\mu^5X^{46}+\mu^{11}X^{44}+\mu^{13}X^{42}+\\&&\mu^9X^{40}+X^{38}+\mu^8X^{36}+
\mu^3X^{34}+\mu^{12}X^{32}+\mu^7X^{30}+\mu^9X^{28}+\mu^8X^{26}+\\&&\mu^{10}X^{24}+
\mu^9X^{22}+\mu^5X^{20}+\mu^2X^{18}+
+\mu^3X^{16}+\mu^2X^{14}+\mu^5X^{12}+\\&&
\mu^8X^{10}+\mu^{11}X^8+\mu^6X^6+\mu^7X^4
\end{array}
$$
$$
\begin{array}{lll}
f_3(X)&=&X^{72}+\mu^{14}X^{68}+\mu^2X^{66}+\mu X^{64}+\mu^4X^{62}+X^{60}
+\mu^9X^{58}+X^{56}+\mu^{14}X^{54}+\\&&\mu X^{52}+X^{50}+\mu^{5}X^{48}+
\mu^{11}X^{44}+\mu^{5}X^{42}+\mu^{9}X^{40}+\mu^{6}X^{38}+\mu^{8}X^{36}+\\&&
\mu^3X^{34}+\mu^{12}X^{32}+\mu^{12}X^{30}+\mu^{9}X^{28}+\mu^{12}X^{26}+
\mu^{10}X^{24}+\mu^{13}X^{22}+\\&&\mu^5X^{20}+\mu^8X^{18}+\mu^3X^{16}+
\mu^2X^{14}+\mu^5X^{12}+\mu^7X^{10}+\mu^{11}X^8+\mu^7X^4;
\end{array}
$$
$$
\begin{array}{lll}
f_4(X)&=&X^{76}+\mu^5X^{74}+\mu^7X^{72}+\mu^3X^{70}+\mu^{9}X^{68}+\mu^{12}X^{66}+
\mu^{6}X^{64}+\mu^{12}X^{62}+\\&&\mu^{3}X^{60}+\mu^{9}X^{58}+\mu^{10}X^{56}+\mu^{12}X^{54}+\mu^{12}X^{52}+\mu^{10}X^{50}+\mu X^{48}+\\&&\mu^{6}X^{46}+
\mu^{5}X^{44}+
\mu^3X^{42}+\mu^{12}X^{40}+\mu^{14}X^{38}+\mu^{13}X^{36}+\mu^{14}X^{34}+\\&&
\mu^3X^{32}+\mu^6X^{30}+\mu^4X^{28}+\mu^13X^{26}+\mu^6X^{24}+
X^22+\mu^{12}X^{20}+\mu^2X^{18}+\\&&\mu^3X^{16}+\mu^{10}X^{14}+\mu^{6}X^{12}+
X^{10}+\mu^{12}X^{6}+\mu^6X^4+\mu^{13}X^2+\mu^9.
\end{array}
$$

{}From MAGMA computation, $\cX$ has genus $17$ and its $2$-rank equals $9$. Further, $\cX({\mathbb{F}}_{16})$, the set of all ${\mathbb{F}}_{16}$-rational points of $\cX$, has size $8$: all {of them} are branches centered at $Y_\infty$, while the ${\mathbb{F}}_{16}$-automorphism group $G$ of $\cX$ is a semi-dihedral group of order $32$ with the unique central involution $u:\,(X,Y)\to(X,Y+1)$.
In particular, $u$ is the unique involution of the cyclic subgroup of $G$ of order $16$ and fixes $\cX({\mathbb{F}}_{16})$ pointwise.  {}From  (\ref{eq2deuring}), $u$  fixes no more points on $\cX$.

The function field of the quotient curve $\bar{\cX}=\cX/\langle u \rangle$ is the subfield $\K(\bar{\cX})=\K(x,z=y^2+y)$ of $\K(\cX)$ and hence $\bar{\cX}$ is a non-singular model of the plane algebraic curve with affine equation
$$Z^2+(f_1(X)+f_2(X))Z+f_4(X)=0.$$
Actually, $\bar{\cX}$ is an elliptic curve. Therefore, the central involution $e$ is not inductive.  

Finally, comparison with Nakajima's bound $|\aut(\cX)|\leq 4(\gamma-1)\leq 32$ shows that $G=\aut(\cX)$. Therefore, if the first hypothesis in (I)  is relaxed to $|S|\geq 2{\gg}-2$, more groups enter in play when an analog of Theorem \ref{princ} is considered.   


\vspace{0,5cm}\noindent {\em Authors' addresses}:

\vspace{0.2 cm} \noindent Massimo GIULIETTI \\
Dipartimento di Matematica e Informatica
\\ Universit\`a degli Studi di Perugia \\ Via Vanvitelli, 1 \\
06123 Perugia
(Italy).\\
 E--mail: {\tt giuliet@dipmat.unipg.it}

\vspace{0.2cm}\noindent G\'abor KORCHM\'AROS\\ Dipartimento di
Matematica\\ Universit\`a della Basilicata\\ Contrada Macchia
Romana\\ 85100 Potenza (Italy).\\E--mail: {\tt
gabor.korchmaros@unibas.it }

    \end{document}